\documentclass{amsart}
\usepackage{amssymb,amsmath,hyperref,multirow,epsfig,float}

\newtheorem{theorem}{Theorem}[section]

\newtheorem{corollary}[theorem]{Corollary}

\newtheorem{lemma}[theorem]{Lemma}
\newtheorem{proposition}[theorem]{Proposition}
\theoremstyle{remark}
\newtheorem{remark}[theorem]{Remark}
\newtheorem{remarks}[theorem]{Remarks}

\numberwithin{equation}{section}

\renewcommand{\subset}{\subseteq}

\newcommand{\A}{\mathbf{A}}

\newcommand{\bs}{\backslash}

\newcommand{\comment}[1]{}
\newcommand{\C}{\mathbf{C}}

\newcommand{\D}{\mathcal{D}}
\newcommand{\ds}{\displaystyle}
\newcommand{\e}{\varepsilon}
\newcommand{\E}{\mathcal{E}}

\newcommand{\F}{\mathcal{F}}

\newcommand{\g}{\gamma}

\newcommand{\GL}{\operatorname{GL}}
\newcommand{\GSp}{\operatorname{GSp}}

\newcommand{\mat}[4]{\begin{pmatrix} {#1} & {#2} \\ {#3} & {#4}
  \end{pmatrix}}

\newcommand{\meas}{\operatorname{meas}}
\renewcommand{\mod}{\text{ mod }}
\newcommand{\mt}{\eta_p^\chi}

\newcommand{\new}{\text{\it new}}

\renewcommand{\O}{\operatorname{O}}

\newcommand{\ol}{\overline}

\newcommand{\PGL}{\operatorname{PGL}}
\newcommand{\Q}{\mathbf{Q}}

\newcommand{\R}{\mathbf{R}}
\renewcommand{\Re}{\operatorname{Re}}

\newcommand{\sg}[1]{\left<{#1}\right>}

\newcommand{\smat}[4]{\bigl(\begin{smallmatrix}{#1}&{#2}\\{#3}&{#4}\end{smallmatrix}\bigr )}

\newcommand{\SL}{\operatorname{SL}}
\newcommand{\SO}{\operatorname{SO}}
\newcommand{\Sp}{\operatorname{Sp}}

\newcommand{\Supp}{\operatorname{Supp}}

\newcommand{\w}{\omega}

\newcommand{\Z}{\mathbf{Z}}

\begin{document}
\title{Weighted distribution of low-lying zeros of $\GL(2)$ $L$-functions}
\author{Andrew Knightly and Caroline Reno}
\address{Department of Mathematics \& Statistics\\University of Maine
\\Neville Hall\\ Orono, ME  04469-5752, USA }
\thanks{This work was partially supported by a grant from the Simons Foundation
  (\#317659 to the first author). Section \ref{eigen} is based in part on the
  University of Maine MA thesis of the second author.}

\begin{abstract} 
We show that if the zeros of an automorphic $L$-function are 
  weight\-ed by the central value of the $L$-function
  or a quadratic imaginary base change,
  then for certain families of holomorphic $\GL(2)$ newforms, 
  it has the effect of changing the distribution type of low-lying 
  zeros from orthogonal to symplectic, for test functions whose Fourier 
  transforms have sufficiently restricted support.
  However, if the $L$-value is twisted by a nontrivial quadratic
  character, the distribution type remains orthogonal.
  The proofs involve two vertical equidistribution results for Hecke
  eigenvalues weighted by central twisted $L$-values.  One of these
    is due to Feigon
      and Whitehouse, and the other
  is new and involves
  an asymmetric probability
  measure that has not appeared in previous equidistribution
  results for $\GL(2)$.
\end{abstract}

\maketitle
\noindent \today
\thispagestyle{empty}

\section{Introduction} \label{intro}

  According to the density conjecture of Katz and Sarnak, 
for any suitable family of $L$-functions, 
  the zeros lying close to the real axis are equidistributed according to
  one of a handful of possible symmetry types coming from compact classical groups
  (\cite{KS1}, \cite{KS2}).
  More precisely, given an $L$-function $L(s,f)$,
  denote its nontrivial zeros by $\rho_f=\frac12+i\g_f$, and define the 1-level density
\[\D(f,\phi)=\sum_{\rho_f} \phi\bigl(\frac{\g_f \log Q_f}{2\pi}\bigr),\]
where $Q_f$ is the analytic conductor of $f$, and $\phi$ is a test function.
 The conjecture predicts that
for any suitable family $\mathcal{F}=\bigcup \mathcal{F}_n$ of automorphic forms,
  with each $\mathcal{F}_n$ finite,
there exists a family $G$ of classical compact groups 
  (being one of O, SO(even), SO(odd), Sp, or U) such that for any 
  even Schwartz function $\phi$ with compactly supported Fourier transform
  $\widehat{\phi}$,
\[\lim_{n\to\infty}\frac{\sum_{f\in \mathcal{F}_n} \D(f,\phi)}
{|\mathcal F_n|} = \int_{-\infty}^\infty \phi(x)W_G(x)dx.\]
Here, $W_G(x)$ is the limiting distribution
  of the 1-level density attached to the eigenvalues of random matrices in $G$
  as the rank tends to $\infty$.
 Of particular relevance to us here are
\[W_{\O}(x) = 1+\frac12\delta_0(x)\]
and
\[W_{\Sp}(x) = 1-\frac{\sin(2\pi x)}{2\pi x},\]
where $\delta_0$ is the Dirac distribution at $0$.  As a distribution,
  $W_{\Sp}(x)$ coincides with $1-\tfrac12\delta_0(x)$ when, as will always
  be the case for us here, $\widehat{\phi}$
  is supported in $(-1,1)$.  This is a consequence of the Plancherel formula
   (\cite[(1.34)]{ILS}).

Averages involving automorphic forms are naturally studied using the trace formula.
  Many variants of the trace formula involve weighting factors, such as the harmonic weight 
  $\frac{|a_f(1)|^2}{\|f\|^2}$ that arises in the Petersson formula.
  In some cases, including that of $\GL(2)$ newforms, the presence of this weight 
  is innocuous in the sense that it does not affect the distribution of low-lying
  zeros, \cite{Mi}.
  However, in the case of zeros of $\GSp(4)$ spinor $L$-functions,
  Kowalski, Saha and Tsimmerman 
  found that the analogous harmonic weight leads to a {\em symplectic}
  distribution, despite a heuristic suggesting that the unweighted distribution 
  is {\em orthogonal}, \cite{KST}.  They gave 
  a striking interpretation of this as evidence for B\"ocherer's conjecture,
  according to which the Fourier coefficient arising in the weight contains 
  arithmetic information in the form of central $L$-values.

The question thus arises: in the simplest case of holomorphic $\GL(2)$ cusp forms,
  if we weight the low-lying zeros by central $L$-values, 
does it likewise change the distribution from orthogonal to symplectic?
The answer depends on the type of $L$-function used in the weight, as we 
  illustrate below using several families with suitably
  restricted test functions. We do not use the Petersson formula, but rather
  the relative trace formulas developed in \cite{FW} and \cite{twists}, in which
  central $L$-values appear directly.

In Theorem \ref{llz}, we consider the effect of weighting by the central
  $L$-value and a Fourier coefficient.  We show for two different families
  of holomorphic newforms that the weighted distribution of low-lying zeros 
  is symplectic
  when $\widehat{\phi}$ is supported in $(-\frac12,\frac12)$.  However,
  if the $L$-value is twisted by a nontrivial quadratic character, the
  weighted distribution is orthogonal.
In Theorem \ref{llz2},
  we show that the zeros of $L$-functions attached 
  to newforms of prime level $N\to\infty$, when weighted by an 
  imaginary quadratic base change central $L$-value, 
  have symplectic distribution for $\widehat{\phi}$ supported in $(-1,1)$.
We do not assume any version of the Generalized Riemann Hypothesis,
   though it motivates the definition of one-level density, and its use can enable one
  to extend the allowable range of support of $\widehat{\phi}$ (\cite{BBDDM},
  \cite{ILS}).
  Of course, it would be of interest to widen the range of support beyond $(-1,1)$
  because the nature of the measure $W_{\Sp}$ changes there.

\begin{theorem}\label{llz}
Let $\chi$ be a primitive real Dirichlet character of modulus $D\ge 1$. 
  Let $r>0$ be an integer relatively prime to $D$.  
    For a holomorphic newform $f(z)=\sum a_f(n)e^{2\pi i nz}$, 
  define the weight
  \begin{equation}\label{w1}
  w_f = \frac{\Lambda(\frac12,f\times \chi)|a_f(r)|^2}{\|f\|^2}
\end{equation}
for the completed $L$-function $\Lambda(s,f\times \chi)$ defined in \eqref{L} below.
Let $\phi$ be any even Schwartz function whose Fourier transform
$\widehat{\phi}(y)=\int_{-\infty}^\infty\phi(x)e^{-2\pi i xy}dx$ is
  supported inside $(-\tfrac12,\tfrac12)$.  Then
\[\lim_{n\to\infty} \frac{\sum_{f\in \mathcal{F}_n} \D(f,\phi)w_f}
  {\sum_{f\in \F_n}w_f}
  =\begin{cases}\ds\int_{-\infty}^\infty \phi(x)W_{\Sp}(x)dx,&\chi\text{ trivial}\\\\
  \ds\int_{-\infty}^\infty \phi(x)W_{\O}(x)dx,&\chi\text{ nontrivial}\end{cases}
\]
in each of the following situations:
\begin{itemize}
\item $n=k$ and $\F_n=F_k(1)$ is the set of newforms of level $1$ with the weight $k$
  ranging over even numbers satisfying $\tau(\chi)^2\neq -i^kD$ 
 for the Gauss sum $\tau(\chi)=\sum_{m=1}^D\chi(m)e^{2\pi i m/D}$.
\item 
 $\F_n=\F_k(N)^{\new}$ (with $N+k\to\infty$ as $n\to\infty$) is the set 
  of newforms of prime level $N\nmid rD$, and even weight $k>2$
chosen so that $\tau(\chi)^2=-i^kD$, or equivalently, 
  $\chi(-1)= -i^k$.
\end{itemize}
\end{theorem}

\begin{remarks}
 (1)   Iwaniec, Luo and Sarnak showed that in the unweighted case,
  the distribution is orthogonal, \cite{ILS}. 

(2) We prove Theorem \ref{llz} in \S\ref{llz1}.  It is shown there that 
  in the second case, if $k$ is fixed and $N\to \infty$, the allowable support of
   $\widehat{\phi}$ can be widened to $[-\alpha,\alpha]$ for any $0<\alpha<1-\tfrac1k$.

(3) The weights $w_f$ are nonnegative by Guo's theorem, \cite{Gu}.
  In \S\ref{llz1}, we also show that the statement of Theorem \ref{llz} remains true
  if we instead use the weight 
  $w_f = \frac{\Lambda(\frac12,f\times \chi)\ol{a_f(r)}}{\|f\|^2}$, which may be negative.
  (Hypotheses on $\mathcal{F}_n$ imply that $a_f(r)$ is real here, though elsewhere in this
  paper it may be complex.)

(4) The conditions involving $\tau(\chi)$ come from the functional
  equation \eqref{fe} when $N=1$.  Since $\chi=\ol{\chi}$, the condition 
$\frac{i^k\tau(\chi)^2}D=-1$ forces the $L$-function 
  to vanish at $s=\frac12$.
  In the first case above (where $N=1$),
  the given condition keeps this from happening,
  and guarantees that the sum of the weights is nonzero when $k$ is sufficiently large.
  In the second case where $N$ is prime, the given condition is desirable since it
  causes the weights attached to the oldforms to vanish, leaving us with an expression
  involving only newforms.
\end{remarks}

\begin{theorem}\label{llz2}
Fix a quadratic discriminant $-D<0$, and let $\chi=\chi_{-D}$ be the 
  associated primitive quadratic Dirichlet character of conductor $D$.
  Let $\F_{N}=\F_k(N)^{\new}$ be the set of holomorphic newforms of prime level $N$
   and fixed even weight $k>2$.
   For $f\in \F_{N}$, define the weight
\[w_f=\frac{\Lambda(\tfrac12,f\times \chi)\Lambda(\tfrac12,f)}{\|f\|^2}.\]
Then for any even Schwartz function $\phi$ with $\widehat{\phi}$
    supported inside $(-1,1)$, we have
\[\lim_{N\to\infty}\frac{\sum_{f\in \F_{N}}\D(f,\phi)w_f}
  {\sum_{f\in\F_{N}}w_f} = \int_{-\infty}^\infty \phi(x) W_{\Sp}(x)dx.\]
Here, $N$ ranges over prime values for which $\chi(-N)=1$.
\end{theorem}
\begin{remark}
    The forms $f$ may in fact be taken to range over
  the family $\F_N^+$ of newforms with epsilon factor $\e_f=1$
  since $\Lambda(\tfrac12,f)=0$ when $\e_f=-1$. The family $\F_N^+$ has
  symmetry type $\SO(\text{even})$ (\cite{ILS}).
\end{remark}

The proof is given in \S \ref{II}.
It uses a special case of the relative trace formula of Ramakrishnan and
Rogawski as extended in \cite{FW} 
by Feigon and Whitehouse. The most general version of their formula (along with the
 recent improvement \cite{FMP} by File, Martin and Pitale)
 could presumably be used to extend the scope of the above theorem.

Theorems \ref{llz} and \ref{llz2} are derived from weighted equidistribution results
  for Hecke eigenvalues at a fixed prime $p$, described in more detail below.
  In each case, the relevant measure is dependent on the
  value $\chi(p)=\pm 1$.
  This dependence plays an interesting role in the proof of the above theorems.
  From the explicit formula, we need to consider the sum over $p$ of the 
  weighted average of the $p$-th Hecke eigenvalue. Because of 
  the nature of the relevant measure, the contribution of the primes satisfying
  $\chi(p)=1$ differs from that of the primes satsifying $\chi(p)=-1$.
  We then apply the prime number theorem for arithmetic progressions to get 
  the results.

In general, the Satake parameters of holomorphic modular forms are known to satisfy
  many equidistribution laws.  Foremost is the celebrated Sato-Tate
  conjecture (proven in \cite{BLGHT}), which asserts that for a fixed non-CM 
  cusp form $f\in S_k(N)$, the sequence of normalized
  Hecke eigenvalues at the unramified primes $p$ (in their natural ordering)
  is equidistributed in $[-2,2]$ relative to the Sato-Tate measure
\begin{equation}\label{STmeas}
d\mu_\infty(x) =\begin{cases}
\frac1\pi\sqrt{1-\frac{x^2}4}\,dx & \text{if }-2\le x\le 2,\\ 
  0 &\text{otherwise}.
 \end{cases}
 \end{equation}

In a different direction, one can fix the prime $p$ and allow the cusp 
  form to vary within a family, possibly with weights.
  In this setting there are strikingly many different equidistribution results 
  for $\GL(2)$ in the literature.\footnote[2]{Some of these have been extended
  to groups of higher rank, e.g., \cite{Z}, \cite{BBR}, \cite{ShT}, \cite{MT}.
  There are also some hybrid results for $\GL(2)$ with both $p$ and 
  the conductor tending to $\infty$, \cite{N}, \cite{W}.}
  We summarize many of these in Table \ref{fig}, giving references
  for the precise statements in each case.
\begin{table}
[H]
\label{fig}
\begin{tabular}{l|c|c|l}
\rule[-2mm]{0mm}{7mm}
Family span& Weights & Measure & References\\
\hline
\rule[-4mm]{0mm}{8mm}
{$S_k(N)\atop{ N+k\to\infty}$} & 1 & 
  $\text{Plancherel:}\atop{\frac{p+1}{(p^{1/2}+p^{-1/2})^2-x^2}\mu_\infty}$
  &\cite{Serre}, \cite{CDF}, \cite{Li2} \\
\hline
\rule[-3mm]{0mm}{8mm}
{$L^2_0(\SL_2(\Z)\bs \mathbf{H})\atop{\lambda_j\le T^2, T\to\infty}$} 
 & 1 & 
  {Plancherel}
  &\cite{Sa} \\
 \hline
\rule[-3mm]{0mm}{8mm}
{$S_k(N)\atop{N+k\to\infty}$} & {$\frac{|a_f(r)|^2}{\|f\|^2}$} & Sato-Tate ($\mu_\infty$)
  &\cite{Li1}, \cite{pethil} \\
\hline
\rule[-3mm]{0mm}{8mm}
{$L^2_0(\Gamma_0(N)\bs \mathbf{H})\atop{N\to\infty}$} 
  & {$\frac{|a_{u_j}(r)|^2h(\lambda_j)}{\|u_j\|^2}$} & Sato-Tate
  &\cite{ftf} \\
\hline
\rule[-3mm]{0mm}{8mm}
{$L^2_0(\SL_2(\Z)\bs \mathbf{H})\atop{\lambda_j\le T^2, T\to\infty}$} 
  & {$\frac{|a_{u_j}(r)|^2}{\|u_j\|^2}$} & Sato-Tate
  &\cite{Br}, \cite{BBR}, \cite{BM} \\
\hline
\rule[-3mm]{0mm}{8mm}
{$S_k(N)^{\text{\new}}\atop{N\to\infty}$} &
\raisebox{0mm}{$\substack{
  \frac{\Lambda(\frac12,f\times\chi)\Lambda(\frac12,f)}{\|f\|^2}\\\\
  \text{for $\chi$ quadratic}}$}&
   $\frac{L_p(\frac12,x,\chi)L_p(\frac12,x)}{L_p(1,\chi)}\mu_\infty$
&
\raisebox{0mm}{$\substack{
\text{\normalsize \cite{RR}, \cite{FW}, \cite{ST},}\\\\
\text{\normalsize Cor. \ref{RRcor} below,}\\\\
 \text{(Also \cite{Su}, \cite{T} for Maass}
  \\ \text{forms of increasing level)}}$}\\
\hline
\rule[-3mm]{0mm}{8mm}
{$S_k(N)\atop{N\to\infty}$} & $\text{regular matrix summation}
  \atop{\text{involving }\tfrac1{\|f\|^2}}$
  & $\tfrac12(1-\frac{x^2}4)^{-1}\mu_\infty$
  & \cite{GMR} \\
\hline
\rule[-3mm]{0mm}{8mm}
{$S_k(N)\atop{N+k\to\infty}$} & 
{$\frac{|a_f(r)|^2\Lambda(s,f\times\chi)}{\|f\|^2}$} 
  & $L_p(s,x,\chi)\mu_\infty$
  & Theorem \ref{main} below \\
\end{tabular}
\caption{Various fixed-$p$ equidistribution results for Hecke eigenvalues
  on $\GL(2)$. (See \eqref{Lx} for the definition of $L_p(s,x,\chi)$.)}
\end{table}

  The last of these examples is new.  Theorem \ref{main} 
  is a generalized and quantitative version of the following.  
  Notation is defined precisely in Section 2.

\begin{theorem}\label{main1}
Let $\chi$ be a primitive real Dirichlet character of conductor $D\ge 1$
   coprime to $N$,
   let $p\nmid DN$ be a fixed prime, and let $\F_{N,k}$ be an orthogonal
  basis for the space $S_k(N)$ of cusp forms, consisting of eigenfunctions
  of the Hecke operator $T_p$, with first Fourier
  coefficient $1$.  Then assuming $N>1$ and $k>2$, the Hecke eigenvalues
  $\lambda_f(p)\in [-2,2]$ for $f\in \F_{N,k}$,
  when weighted by the central twisted $L$-values
  $w_f=\frac{\Lambda(\frac{1}2,f\times\chi)}{\|f\|^2}$,
  become equidistributed in $[-2,2]$ with respect to the probability measure
\[d\mu_p(x)= \frac{p}{(p+1)-x\chi(p)\sqrt p}\,d\mu_\infty(x)\]
  as $N+k\to\infty$.
\end{theorem}

  We emphasize that $\chi$ is allowed to be trivial. 
  In the generalized version (Theorem \ref{main}), 
  $\chi$ need not be real, and we do not specialize the $L$-value in $w_f$ to $s=\frac12$.

There is a natural interpretation of the measure appearing in the above 
  theorem.  See the remark after Theorem \ref{main}.
  Interestingly, the measure is not symmetric, though as expected
  it converges to the Sato-Tate measure as $p\to \infty$. It 
  is plotted below in the case $p=5$ when $\chi(5)=1$:

\centerline{\includegraphics[height=2cm]{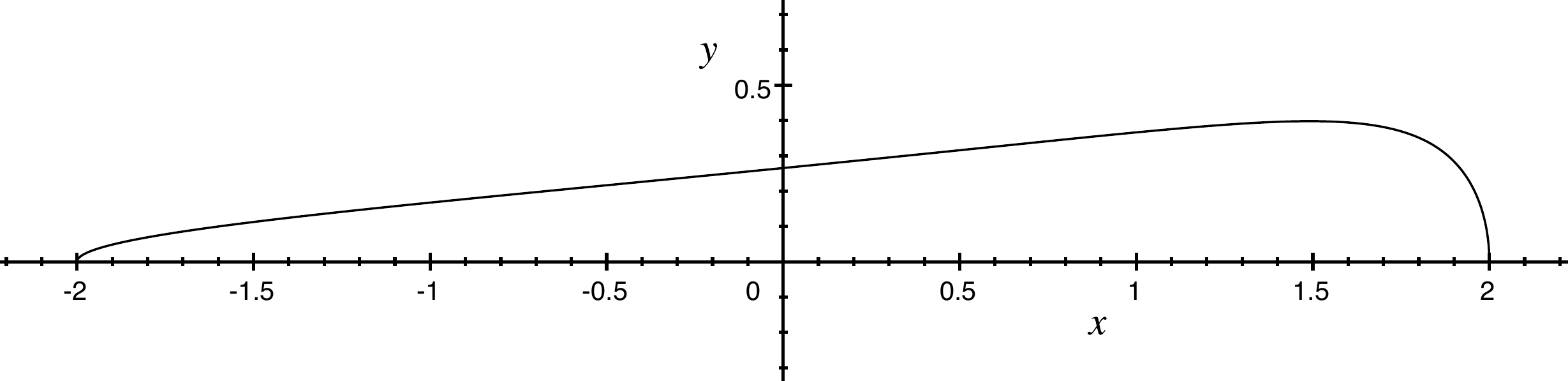}}

\noindent If $\chi(p)=-1$, there is an analogous negative bias.

We also give, in Corollary \ref{RRcor}, another result of this nature, namely
  that for newforms $f\in S_k(N)$ with $N$ prime,
 the $\lambda_f(p)$, when weighted as in Theorem \ref{llz2},
  become equidistributed in the limit as $N\to\infty$
   relative to the measure $\mt$ given in the sixth row of the above table.
This is essentially the main result of \cite{RR}.
  We obtain a more general statement by keeping track of
  the dependence on $k$ in their calculations.
   The measure $\mt$ depends on $\chi$.  It 
  exhibits a similar positive bias precisely when $\chi(p)=1$.
When $\chi(p)=-1$, it coincides with the Plancherel measure, which is even.

\section{Preliminaries on modular forms}

Fix a Dirichlet character $\psi$
   modulo $N$, and let $S_k(N,\psi)$ be the space of holomorphic cusp forms 
  $f$ on the complex upper half-plane $\mathbf{H}$ that transform under
  the action of
  $\Gamma_0(N)=\{\smat abcd\in \SL_2(\Z)|\, c\in N\Z\}$ according to
\[f(\frac{az+b}{cz+d})={\psi(d)}(cz+d)^kf(z).\]
We normalize the Petersson inner product on $S_k(N,\psi)$ by
\begin{equation}\label{petnorm}
\|f\|^2=\frac1{\nu(N)}\iint_{\Gamma_0(N)\bs \mathbf H}|f(z)|^2y^k\frac{dx\,dy}{y^2},
\end{equation}
where
\[\nu(N)=[\SL_2(\Z):\Gamma_0(N)].\]

For us, a {\em Hecke eigenform} is a simultaneous eigenfunction of the 
  Hecke operators 
\[T_n f(z) = n^{k-1}\sum_{ad=n,\atop{a>0}}\sum_{b=0}^{d-1}\psi(a)d^{-k}f(\frac{az+b}d)\]
  for $(n,N)=1$, normalized to have first Fourier coefficient $1$.
  Given a Hecke eigenform 
\[f(z)=\sum_{n>0}a_f(n)q^n \qquad(q=e^{2\pi i z}),\]
 for a prime $p\nmid N$ we fix a complex square root $\psi(p)^{1/2}$ 
and define the normalized $p$-power Hecke eigenvalue
\begin{equation}\label{he}
\lambda_f({p^\ell})= \frac{a_f({p^\ell})}{\psi(p)^{\ell/2}p^{\ell(k-1)/2}}
   \quad(\ell\ge0).
\end{equation}
  By Deligne's theorem $\lambda_f(p)\in [-2,2]$, and our interest is in the distribution
  of these numbers as $f$ varies.  For any integer $\ell\ge 0$, 
\[\lambda_f({p^\ell}) = X_\ell(\lambda_f(p)),\]
  where $X_\ell$ is the Chebyshev polynomial of degree $\ell$ defined by
  $X_\ell(2\cos\theta)=\frac{\sin((\ell+1)\theta)}{\sin\theta}$
  (see, e.g., \cite[Prop. 29.8]{KL}, where $\w'$ corresponds to $\psi^{-1}$).
    Equivalently,
\begin{equation}\label{Cheb}
a_f({p^\ell}) = \psi(p)^{\ell/2}p^{\ell(k-1)/2}X_\ell(\lambda_f(p)).
\end{equation}

Fix an integer $D$ with $(D,N)=1$, and let $\chi$ be a primitive Dirichlet
  character modulo $D$.
The $\chi$-twisted $L$-function of $f$ is given for $\Re(s)>1$
   by the Dirichlet series
\[L(s,f\times\chi)=\sum_{n>0}\frac{\chi(n)a_f(n)}{n^{s+\frac{k-1}2}}.\]
The completed $L$-function
\begin{equation}\label{L}
\Lambda(s,f\times\chi)=(2\pi)^{-s-\frac{k-1}2}\Gamma(s+\tfrac{k-1}2)L(s,f\times \chi)
\end{equation}
has an analytic continuation to the complex plane and 
  satisfies a functional equation relating $s$ to $1-s$, which takes the form
\begin{equation}\label{fe}
\Lambda(s,f\times \chi)=\frac{i^k}{D^{2s-1}}
  \frac{\tau(\chi)^2}D\Lambda(1-s,f\times \ol{\chi})
\end{equation}
when $N=1$.  Here, $\tau(\chi)=\sum_{m=1}^D\chi(m)e^{2\pi i m/D}$ is
 the Gauss sum attached to $\chi$.

Given $x\in [-2,2]$ and $p\nmid DN$, there is a unique unramified unitary representation
  $\pi_{x,p}$ of $\GL_2(\Q_p)$ with Satake parameters $\alpha_p,\beta_p$
  satisfying $\alpha_p+\beta_p=x\psi(p)^{1/2}$ and $\alpha_p\beta_p=\psi(p)$.
   We denote its twisted $L$-factor by
\begin{equation}\label{Lx}
L_p(s,x,\chi)=
(1-x\psi(p)^{1/2}\chi(p)p^{-s}+\psi(p)\chi(p)^2p^{-2s})^{-1}.
\end{equation}
With this notation, the local $L$-factor of $L(s,f\times\chi)$ is
\[L_p(s,f\times\chi)=L_p(s,\lambda_f(p),\chi).\]

\section{Weighted equidistribution of Hecke eigenvalues I}\label{eigen}

Fix a weight $k>2$ and a level $N>1$, and let 
\[\mathcal{F}=\F_{N,k}=\mathcal{F}_k(N,\psi) \]
  be an orthogonal basis for $S_k(N,\psi)$ consisting of Hecke eigenforms.
Fix $D$ and $\chi$ as above, and fix an integer $r$ relatively prime to $D$.
  In this section, we do not assume that $\chi^2=1$ unless explicitly stated.
For each $f\in\F$, define the (complex) weight
\begin{equation}\label{wh}
w_f =\frac{\ol{a_f(r)}\Lambda(s,f\times \chi)}{\|f\|^2}.
\end{equation}
  Then for all $s=\sigma+i\tau$ in the strip $1-\tfrac{k-1}2<\sigma<\tfrac{k-1}2$
  and all integers $n$ relatively prime to $DN$,
by Theorem 1.1 of \cite{twists} (which is a twisted version of the main theorem
  of \cite{petrr}), we have
\begin{align}\label{JKtf}
\frac1{\nu(N)}
\sum_{f\in\mathcal{F}}
  w_f a_f({n})
 =&\frac{2^{k-1}(2\pi rn)^{\frac{k-1}2-s}}{(k-2)!}\Gamma(s+\tfrac{k-1}2)
  \sum_{d|(n,r)} d^{2s}\psi(\tfrac nd)\chi(\tfrac{rn}{d^2})\\
\notag &+O\left(\gcd(n,r)\frac{(4\pi rn)^{k-1}D^{\frac{k}2-\sigma}\varphi(D)}
{N^{\sigma+\frac{k-1}2}(k-2)!}\right).
\end{align}
(We have adjusted for the fact that in \cite{twists} the $L$-function is 
  normalized to have central point $\tfrac k2$, whereas here the 
  central point is $\tfrac12$.)
  The implied constant is explicit in \cite{twists}, and depends only on $s$.

  Now fix a prime $p\nmid rND$.
  Taking $n=p^\ell$ and substituting \eqref{Cheb}, the above becomes
\begin{equation}\label{ellsum}
\frac1{\nu(N)}
\sum_{f\in\mathcal{F}} w_f X_\ell(\lambda_f(p)) = F_\ell + E_\ell,
\end{equation}
where
 \begin{equation}\label{Fell}
F_\ell=\left(\psi(p)^{1/2}\chi(p)p^{- s}\right)^{\!\ell}
  \,\frac{2^{k-1}(2\pi r)^{\frac{k-1}2-s}\chi(r)}{(k-2)!}
  \Gamma(s+\tfrac{k-1}2),
\end{equation}
and $E_\ell$ is an error term satisfying
\begin{equation}\label{El}
E_\ell \ll 
p^{\frac{\ell(k-1)}2}\frac{(4\pi r)^{k-1}D^{\frac{k}2-\sigma}\varphi(D)}{N^{\sigma+\frac{k-1}2}(k-2)!}.
\end{equation}

\begin{proposition}\label{Xl}
For any $\ell\ge 0$ and $0<\sigma<1$,
\begin{equation}\label{lim}
\frac{\sum_{f\in\mathcal{F}_k(N,\psi)} w_f X_\ell(\lambda_f(p))}
{\sum_{f\in\mathcal{F}_k(N,\psi)} w_f } 
  =\left[\psi(p)^{1/2} \chi(p)p^{-s}\right]^\ell
+O\left( \frac{p^{\frac{\ell (k-1)}2}(4\pi rDe)^{k/2}}{N^{\frac{k-1}2}k^{\frac{k}2-1}}\right),
\end{equation}
where the implied constant depends only on $r,s,D$.
\end{proposition}
\begin{remarks}\label{Xlrem}
    (1)
  When $N>1$, it is shown in \cite[\S9]{twists} that the sum of the weights
  is nonzero when $N+k$ is sufficiently large. 
  When $N=1$, this can only be verified under
  certain extra conditions mentioned in Theorem \ref{main} below.

(2) By taking $n=rp^\ell$ in \eqref{JKtf} rather than $n=p^\ell$,
   and using $a_f(rp^\ell)=a_f(r)a_f(p^\ell)$,
  one obtains \eqref{ellsum} with the different weight 
  \begin{equation}\label{wf2}
w_f=\frac{\Lambda(s,f\times \chi)|a_f(r)|^2}{\|f\|^2}.
\end{equation}
 In \eqref{Fell}
  we then have to replace $(2\pi r)$ by $(2\pi r^2)$, and $\chi(r)$ by $\sum_{d|r}d^{2s}\psi(\tfrac rd)
  \chi(\tfrac{r^2}{d^2})$; in \eqref{El} $r^2$ replaces $r$, and
  one additional factor of $r$ is needed due to $\gcd(n,r)=r$.
  As long as the above sum over $d$ is nonzero
(for example, if $\chi^2$ and $\psi$ are trivial and $s$ is real), \eqref{lim} 
  holds with the alternative weight upon replacing $r$ by $r^2$ in the error term.
\end{remarks}

\begin{proof}[Proof of Proposition \ref{Xl}]
In the notation of \eqref{ellsum}, the left-hand side of \eqref{lim} is
\[\frac{F_\ell+E_\ell}{F_0+E_0} = \frac{F_\ell}{F_0} 
  +\frac{E_\ell -\frac{F_\ell}{F_0} E_0}{F_0+E_0}.\]
This will immediately give \eqref{lim} once we show that the second term on the
  right-hand side has the desired rate of decay.  If we denote the right-hand side of
  \eqref{El} by $p^{\frac{\ell(k-1)}2}C_0$, then
\[  \frac{E_\ell -\frac{F_\ell}{F_0} E_0}{F_0+E_0}
=\frac{E_\ell -\psi(p)^{\ell/2}\chi(p)^\ell p^{-\ell s}E_0}{F_0+E_0}
\ll\frac{(p^{\frac{\ell(k-1)}2} +p^{-\ell \sigma})C_0}{F_0+E_0}\]
\[\ll
\frac{p^{\frac{\ell(k-1)}2}C_0}{F_0+E_0}
  =p^{\frac{\ell(k-1)}2}\frac{\frac{C_0}{F_0}}{1+\frac{E_0}{F_0}}.\]
In \S9 of \cite{twists} (taking $n=1$), it is shown that
\[\frac{E_0}{F_0}\ll\frac{C_0}{F_0}\ll \frac{(4\pi rDe)^{k/2}}{N^{(k-1)/2}k^{{k}/2-1}},\]
where the implied constant depends on $r,s,D.$  The proposition follows.
\end{proof}

Define a measure
\begin{equation}\label{mu1}
d\mu_{p,s,\chi}(x)=\sum_{\ell=0}^\infty 
  \left[\psi(p)^{1/2} \chi(p)p^{-s}\right]^\ell X_\ell(x)\, d\mu_\infty(x),
\end{equation}
where as before, $\mu_\infty$ is the Sato-Tate measure on $\R$ with support
  $[-2,2]$, and $X_\ell$ is the Chebyshev polynomial.
  The infinite series is absolutely convergent provided $|x|\le 2$ and 
  $\Re(s)>0$.  Indeed, if $|x|\le 2$ and $|t|<1$, we have
  the well-known identity
\begin{equation}\label{Xlid}
\sum_{\ell=0}^\infty t^\ell X_\ell(x) = \frac1{1-xt+t^2}.
\end{equation} 
Therefore
\[d\mu_{p,s,\chi}(x)
= \frac1{1-x\psi(p)^{1/2}\chi(p)p^{-s}+\psi(p)\chi(p)^2p^{-2s}}d\mu_\infty(x).\]
As pointed out to us by Fan Zhou, this gives (in the notation of \eqref{Lx})
\begin{equation}\label{mu2}
d\mu_{p,s,\chi}(x)=L_p(s,x,\chi)d\mu_\infty(x).
\end{equation}
The above is a complex-valued probability measure since, by
  \eqref{mu1} and the orthonormality of the $X_\ell(x)$ relative to $\mu_\infty$,
  $\int X_0(x) d\mu_{p,s,\chi}=1.$ 
We note that when $s=\tfrac12$,  $\psi$ is trivial, $\psi(p)^{1/2}$ is chosen to be $1$,
   and $\chi$ is real,
\[d\mu_{p,\frac12,\chi}(x)=\frac{p}{(p+1)-x\chi(p)\sqrt{p}}\,d\mu_\infty(x)\]
is the measure given in Theorem \ref{main1}.

\begin{theorem}\label{main}
Fix $s$ in the critical strip $0<\Re(s)<1$, let $N>1$ be coprime to $rD$, let $k>2$, 
  let $\psi$ be a Dirichlet character whose conductor divides $N$, fix a prime
  $p\nmid rND$, and a choice of square root $\psi(p)^{1/2}$.  Define weights $w_f$
  as in \eqref{wh} and Hecke eigenvalues $\lambda_f(p)$ as in \eqref{he}.
Then the $\lambda_f(p)$ for $f\in \F_k(N,\psi)$ become
  $w_f$-equidistributed in $[-2,2]$ relative to the measure $\mu_{p,s,\chi}$
  as $N+k\to\infty$.  In other words, for any continuous function $\phi$ 
  on $\R$,
\begin{equation}\label{ed}
\lim_{N+k\to\infty}
\frac{\sum_{f\in\mathcal{F}_k(N,\psi)} w_f \phi(\lambda_f(p))}
{\sum_{f\in\mathcal{F}_k(N,\psi)} w_f } =\int_\R \phi \,d\mu_{p,s,\chi}.
\end{equation}
Moreover, if $\phi$ is a polynomial of degree $d$, then
\begin{equation}\label{ed2}
\frac{\sum_{f\in\mathcal{F}_k(N,\psi)} w_f \phi(\lambda_f(p))}
{\sum_{f\in\mathcal{F}_k(N,\psi)} w_f } =\int_\R \phi \,d\mu_{p,s,\chi}
+ O\left( \frac{p^{\frac{d(k-1)}2}(4\pi rDe)^{k/2}}
  {N^{\frac{k-1}2}k^{\frac{k}2-1}}\|\phi\|_{ST}\right),
\end{equation}
where $\|\phi\|_{ST}$ is the norm of $\phi$ in $L^2(\R,\mu_\infty)$.

When $N=1$, the equidistribution assertion \eqref{ed} still holds, 
  provided $\chi^2=1$,  $s=\tfrac12$ and
  $\frac{i^k\tau(\chi)^2}D\neq -1$.

Lastly, if $\chi$ is quadratic, $\psi$ is trivial, and $s$ is real,
  all of the above statements hold if instead we use the
  nonnegative weights given in \eqref{w1} and we replace $r$ by
 $r^2$ in the error term of \eqref{ed2}.
\end{theorem}

\begin{remark}
  The measure $\mu_{p,s,\chi}$ appearing here is natural for the following reason.
  The weight $w_f$ depends directly on $\lambda_f(p)$ via the
  local $L$-factor $L_p(s,f\times\chi)=L_p(s,\lambda_f(p),\chi)$ (in the
  notation of \eqref{Lx}).  Assuming the remaining $L$-factors do not affect
  the distribution of the $\lambda_f(p)$, on the left-hand side of \eqref{ed}
  we have something resembling a Petersson-weighted average
 of the function $L_p(s,x,\chi)\phi(x)$ at the points $\lambda_f(p)$, which, in view
  of the equidistribution result \cite{Li1}, tends to the integral of this function
  against the Sato-Tate measure. By \eqref{mu2}, this is exactly what appears on the 
  right-hand side of \eqref{ed}.
\end{remark}

\begin{proof}
First take $N>1$.  By the fact that the Chebyshev polynomials are orthonormal 
  relative to the Sato-Tate measure $\mu_\infty$, we see from \eqref{mu1}
   that \eqref{lim} gives \eqref{ed} with $\phi=X_\ell$ for $\ell\ge 0$.
    By linearity it holds if $\phi$ is any polynomial, so
  by Weierstrass approximation, \eqref{ed} holds for all continuous functions.

Since $\|X_\ell\|_{ST}=1$ for all $\ell$, Proposition \ref{Xl} gives
  \eqref{ed2} when $\phi=X_\ell$.  For an arbitrary polynomial $\phi$ of degree $d$,
  we may write $\phi=\sum_{\ell=0}^d \sg{\phi,X_\ell}X_\ell$, so denoting the
  left-hand side of \eqref{ed2} by $\E(\phi)$, we have
\[\left|\E(\phi)-\int\phi \,d\mu_{p,s,\chi}\right|=\left|\sum_{\ell=0}^d
   \sg{\phi,X_\ell}\left(\E(X_\ell)-\int X_\ell \,d\mu_{p,s,\chi}\right)\right|\]
  Applying \eqref{mu1}, \eqref{lim}, and the Schwarz inequality 
  $|\sg{\phi,X_\ell}|\le \|\phi\|_{ST}$,
  the above is
\[\ll \|\phi\|_{ST}\frac{(4\pi rDe)^{k/2}}{N^{\frac{k-1}2}k^{\frac k2-1}}
  \sum_{\ell=0}^dp^{\frac{\ell(k-1)}2},\]
and \eqref{ed2} follows.

Now suppose $N=1$, $\chi^2=1$, and $s=\frac 12$.  Then there is an extra main term in
  \cite[Theorem 1.1]{twists}, so that in place of \eqref{Fell}, we have
\[F_\ell = (\chi(p)p^{-1/2})^\ell\,\frac{2^{k-1}(2\pi r)^{\frac{k}2-1}\chi(r)}{(k-2)!}
  \Gamma(\tfrac{k}2)\left[1+i^k\frac{\tau(\chi)^2}{D}\right].
\]
(The extra main term contains the factor $\ol{\chi(r p^\ell)}$, 
  so we we have imposed $\chi^2=1$ to make this equal to $\chi(p)^\ell\chi(r)$.)
The rest of the argument then goes through as above,
   provided the bracketed expression is nonzero.

Finally, if the alternative nonnegative weights \eqref{w1} are used, then
    in view of Remark \ref{Xlrem}(2), everything goes through as above.
\end{proof}

\section{Low-lying zeros I}\label{llz1}

In this section we derive Theorem \ref{llz} from the results of the previous section
   by standard methods (see, for example, \cite[\S9]{Ko}).
We will use Proposition \ref{Xl}, together with 
the following consequence of the explicit formula for the $L$-function
  of a holomorphic newform $f\in \mathcal{F}_k(N)^{\new}$ with analytic conductor $Q_f=k^2N$:
\begin{align}\label{D}\D(f,\phi)=&\widehat{\phi}(0)+\frac12\phi(0) 
-2\sum_{p\nmid N} \frac{\lambda_f(p)\log p}{p^{1/2}\log Q_f}
  \widehat{\phi}\Bigl(\frac{\log p}{\log Q_f}\Bigr)\\
\notag & -2\sum_{p\nmid N} \frac{\lambda_f(p^2)\log p}{p\log Q_f}
  \widehat{\phi}\Bigl(\frac{2\log p}{\log Q_f}\Bigr)
+O\Bigl(\frac{\log\log 3N}{\log Q_f}\Bigr).
\end{align}
This holds for any even Schwartz function $\phi$ on $\R$ whose Fourier transform
  has compact support, \cite[Lemma 4.1]{ILS}.

For the remainder of this section, $\chi$ is a real Dirichlet character, and
  $\mathcal{F}$ denotes one of the following families given in Theorem \ref{llz}:
\begin{enumerate}
\item $\F=\mathcal{F}_k(1)$, the set of Hecke eigenforms of
  level $N=1$ and even weight $k$ chosen so that $\frac{i^k\tau(\chi)^2}D\neq -1$.
\item $\F=\mathcal{F}_k(N)^{\new}$, where $N\nmid rD$ is prime.  In this case, 
  the even weight $k\ge 4$ is chosen so that $\frac{i^k\tau(\chi)^2}D = -1$.
\end{enumerate}

We need to consider the weighted average of $\D(f,\phi)$ 
  over $\mathcal{F}$.
 To simplify notation, given a function $A:f\mapsto A_f$ on $\mathcal{F}$,
  we define the $w$-weighted average of $A$ by
\[\E_{\mathcal{F}}^w(A)=\frac{\sum_{f\in \F}A_fw_f}{\sum_{f\in\F}w_f},\]
where, for all $f$ we take $w_f$ to be either the weight defined in \eqref{wh} with
   $s=\tfrac12$, or the weight defined in \eqref{w1}.
  In the latter case, Remark \ref{Xlrem}(2) should be borne in mind
  for the remainder of this section.

  When $N=1$
  and $\frac{i^k\tau(\chi)^2}D=-1$, or equivalently, $\chi(-1)=-i^k$,
 the functional equation \eqref{fe}
  forces $\Lambda(\tfrac12,f\times\chi)=0$ since $\chi$ is real.
  Hence when $N$ is prime, the conditions imposed on 
  $k$ and $\chi$ ensure that $w_f=0$ for all Hecke eigenforms $f$ of level $1$ and weight $k$.
    If we set $f_N(z)=f(Nz)$ for such $f$, we have
  $\Lambda(\tfrac12,f_N\times \chi)=\frac{\chi(N)}{N^{{k}/2}}\Lambda(\frac12,f\times\chi)=0$
  as well, so that $w_h=0$ for all $h$ in the span of $\{f,f_N\}$.
  Therefore
\begin{equation}\label{FF}
\E_{\mathcal{F}}^w(A)=\E_{\F_k(N)}^w(A)
\end{equation}
in this case, i.e., the value is unaffected if we average over an orthogonal basis for
  the full space $S_k(N)$, rather than restricting to newforms.

Since $Q_f=k^2N$ is constant across $\mathcal{F}$, we denote it by $Q$ in what follows.  
  By \eqref{D}, we have
\begin{align}
\notag \frac{\sum_{f\in \F} \D(f,\phi)w_f}
  {\sum_{f\in \F}w_f}
= \,& \widehat{\phi}(0)+\frac12\phi(0) 
+O\Bigl(\frac{\log\log 3N}{\log Q}\Bigr)
\\
\label{p1}
&-2\sum_{p\nmid N} \frac{\E_{\F}^w(\lambda_\cdot(p))\log p}{p^{1/2}\log Q}
  \widehat{\phi}\Bigl(\frac{\log p}{\log Q}\Bigr)\\
\label{p2}
& -2\sum_{p\nmid N} \frac{\E_{\F}^w(\lambda_\cdot(p^2))\log p}{p\log Q}
  \widehat{\phi}\Bigl(\frac{2\log p}{\log Q}\Bigr).
\end{align}
Taking $s=\tfrac12$, $\psi$ trivial, and $\ell=1,2$ in \eqref{lim},
   we have (using \eqref{FF} when $N$ is prime) 
\begin{equation}\label{p1b}
\E_{\F}^w(\lambda_\cdot(p)) = \chi(p)p^{-1/2} 
  + O\Bigl(\frac{p^{\frac{k-1}2}R^k}{N^{\frac{k-1}2}k^{\frac k2-1}}\Bigr),
\end{equation}
and
\begin{equation}\label{p2b}
\E_{\F}^w(\lambda_\cdot(p^2)) = \chi(p)^2p^{-1} + 
  O\Bigl(\frac{p^{k-1}R^k}{N^{\frac{k-1}2}k^{\frac k2-1}}\Bigr)
\end{equation}
for a positive constant $R$ depending on $D$ and $r$.
It is a consequence of the prime number theorem that for any real number $m>-1$,
\[\sum_{p\le x} p^m\log p\sim \frac{x^{m+1}}{m+1}\]
as $x\to\infty$.  
If the support of $\widehat{\phi}$ is contained in $[-\alpha,\alpha]$, the sum in \eqref{p1}
  is restricted to $p\le Q^\alpha$.
Therefore, the contribution to \eqref{p1} of the error term in \eqref{p1b} is
\[ \ll \frac{R^k}{N^{\frac{k-1}2}k^{\frac k2-1}}
\sum_{p\le Q^\alpha} p^{\frac k2-1}\log p
\ll \frac{Q^{\frac{\alpha k}2}R^k}{(\frac k2)N^{\frac{k-1}2}k^{\frac k2-1}}
= \frac{2N^{\frac{\alpha k}2}k^{\alpha k}R^k}{N^{\frac{k-1}2}k^{\frac k2}}\]
\[
 = O\Bigl(\frac{1}{\log Q}\Bigr),\]
provided $\alpha <\tfrac 12$.  (If $k$ is fixed, we only need $\alpha<1-\tfrac1k$.)
The contribution to \eqref{p2} of the error term in \eqref{p2b} is
\[ \ll \frac{R^k}{N^{\frac{k-1}2}k^{\frac k2-1}}
\sum_{p\le Q^{\alpha/2}} p^{k-2}\log p
\ll \frac{Q^{\frac{\alpha(k-1)}2}R^k}{(k-1)N^{\frac{k-1}2}k^{\frac k2-1}}
\ll \frac{N^{\frac{\alpha(k-1)}2 }k^{\alpha(k-1)}R^k}{N^{\frac{k-1}2}k^{\frac k2}},\]
which may likewise be absorbed into the error term if $\alpha<\tfrac12$.

It remains to treat the contribution of the main terms of \eqref{p1b} and \eqref{p2b}
  to \eqref{p1} and \eqref{p2} respectively.  
  If $\chi$ is trivial, the former yields
\[-2\sum_{p\nmid N} \frac{\log p}{p\log Q}
  \widehat{\phi}\Bigl(\frac{\log p}{\log Q}\Bigr)
=-2\sum_{p} \frac{\log p}{p\log Q}
  \widehat{\phi}\Bigl(\frac{\log p}{\log Q}\Bigr)+O\Bigl(\frac{\log\log 3N}{\log Q}\Bigr)\]
(by (4.19$'$) of \cite{ILS}), which in turn is
\[  =-\phi(0)+O\Bigl(\frac{\log\log 3N}{\log Q}\Bigr)
\]
by the prime number theorem, using the fact that $\phi$ is even.

On the other hand, if $\chi$ is nontrivial,
   then the main term of \eqref{p1b} contributes
\[-2\sum_{p:\chi(p)=1} \frac{\log p}{p\log Q}
  \widehat{\phi}\Bigl(\frac{\log p}{\log Q}\Bigr)
+2\sum_{p:\chi(p)=-1} \frac{\log p}{p\log Q}
  \widehat{\phi}\Bigl(\frac{\log p}{\log Q}\Bigr)
+O\Bigl(\frac{\log\log 3N}{\log Q}\Bigr).\]
The value of $\chi$ is $1$ on
  exactly half of the primes.
By the prime number theorem for arithmetic progressions, the above is
\[=-\frac12\phi(0)+\frac12\phi(0) 
+O\Bigl(\frac{\log\log 3N}{\log Q}\Bigr)
=O\Bigl(\frac{\log\log 3N}{\log Q}\Bigr).\]

Lastly, for any real $\chi$, the contribution of
  the main term of \eqref{p2b} is
\begin{equation}\label{Ep2}
    \ll 2\sum_{p\nmid N} \frac{\log p}{p^2\log Q}
  |\widehat{\phi}\Bigl(\frac{2\log p}{\log Q}\Bigr)|=O\Bigr(\frac1{\log Q}\Bigr).
\end{equation}

Putting everything together, we conclude that when $\alpha<\tfrac12$,
\begin{equation}\label{final}
\frac{\sum_{f\in \F} \D(f,\phi)w_f}
  {\sum_{f\in \F}w_f}
=\begin{cases} \widehat{\phi}(0)-\frac12\phi(0) 
+O\Bigl(\frac{\log\log 3N}{\log (k^2N)}\Bigr), &\chi \text{ trivial}\\
\widehat{\phi}(0)+\frac12\phi(0)
+O\Bigl(\frac{\log\log 3N}{\log (k^2N)}\Bigr), &\chi \text{ nontrivial.}
\end{cases}
\end{equation}
which proves Theorem \ref{llz}.

\section{Weighted equidistribution of Hecke eigenvalues II}\label{eigen2}

We recall the setup from Theorem \ref{llz2}:
 $-D<0$ is the discriminant of a quadratic field $E=\Q[\sqrt{-D}]$,
  $\chi=\chi_{-D}$
   is the associated primitive quadratic character modulo $D$ given
    by the Kronecker symbol $n\mapsto \left(\tfrac {-D}n\right)$,
  and $N$ is a prime number for which
  $\chi(-N)=1$.   The latter condition means that $N$ is inert in $E$.
  For $k>2$ even, we let $\F=\F_{N,k}^{\new}$ be the set of 
  holomorphic newforms of weight $k$ and level $N$.  For $f\in \F$, we 
  define the weight
\begin{equation}\label{wf3}
    w_f=\frac{\Lambda(\tfrac12,f\times \chi)\Lambda(\tfrac12,f)}{\|f\|^2}
=\frac{\Lambda(\tfrac12,f_E)}{\|f\|^2},
\end{equation}
where $f_E$ is the base change of $f$ to $E$.

\begin{proposition}\label{RR}
With hypotheses as above,
for any $\ell\ge 0$, and any prime $p\nmid ND$,
\begin{equation}\label{Ew}
\E^w_\F(\lambda_\cdot(p^\ell)):=
  \frac{\sum_{f\in\F} w_f X_\ell(\lambda_f(p))}{\sum_{f\in\F}
  w_f}= \int_{\R}X_\ell\, d\mt+O\Bigl(\frac{p^{\ell(k+\frac{1}2)}D^{k}}
  {k^{1/2}N^{k/2-\e}}\Bigr),
\end{equation}
where 
\[\mt(x)=\begin{cases}\ds\frac{p+1}{(p^{1/2}+p^{-1/2})^2-x^2}\mu_\infty(x)
  &\text{if }\chi(p)=-1\\
  \ds\frac{p-1}{(p^{1/2}+p^{-1/2}-x)^2}\mu_\infty(x)&\text{if }\chi(p)=1,
\end{cases}\]
and the implied constant depends only on $\chi$, $\ell$, $D$, and $\e\in (0,1)$.
    Furthermore, if $N> p^\ell D$, then \eqref{Ew} holds with {\em no} error term:
    \begin{equation}\label{Ew2}
\E^w_\F(\lambda_\cdot(p^\ell))=
   \int_{\R}X_\ell\, d\mt.
    \end{equation}
\end{proposition}
  Equation \eqref{Ew} is essentially the main result of \cite{RR}, but we have divided by
  the sum of the weights, and shown the dependence
  on $p$ and $k$ explicitly in the error term.
  The proof is somewhat involved, so we defer it to Section \ref{RRthm}.
  Equation \eqref{Ew2} likewise follows from a special case of \cite[Theorem 6.1]{FW}.
  Details are provided in Section \ref{FWthm}.

\begin{corollary}\label{RRcor}
Assume the hypotheses above. 
 Then the multiset $\{\lambda_f(p)|\,f\in\F_{N,k}^{\new}\}$ of normalized
   $p$-th Hecke eigenvalues, when weighted as above, becomes 
  equidistributed in $[-2,2]$ with respect to the measure $\mt$ as $N\to\infty$.
  Thus, for any continuous function $\phi$,
\begin{equation}\label{RRlim}
\lim_{N\to\infty}
  \frac{\sum_{f\in\F_{N,k}^{\new}} w_f\, \phi(\lambda_f(p))}{\sum_{f\in\F_{N,k}^{\new}}w_f} 
=\int_{\R}\phi\,d\mt.
\end{equation}
Moreover, if $\phi$ is a polynomial of degree $d$, then
  \begin{equation}\label{RRq}
\frac{\sum_{f\in\F} w_f\, \phi(\lambda_f(p))}{\sum_{f\in\F}
  w_f}= \int_{\R}\phi\, d\mt
+O\Bigl(\|\phi\|_{ST}\frac{p^{d(k+\frac{1}2)}D^{k}} {k^{1/2}N^{k/2-\e}}\Bigr).
\end{equation}
      The error term in \eqref{RRq} vanishes if $N> p^d D$.
\end{corollary}
\begin{remarks}
    (1) In Theorem A of \cite{RR}, a much stronger claim is made,
  namely that in \eqref{RRq}, by Weierstrass approximation we can take $\phi$ to be the 
  characteristic function of any subinterval of $[-2,2]$,
  preserving the error term $O(N^{-k/2+\e})$.
  However, because the error in \eqref{RRq} depends in a crucial way on 
  the approximating polynomial $\phi$,
  their argument is incomplete.
  Possibly one could use the method of Murty and 
  Sinha \cite{MS}, but we have not investigated this.

(2) Because of \eqref{Ew2}, the weight $k$ may vary in any fashion as $N\to \infty$.
  However we cannot obtain the conclusion
  for fixed $N$ and $k\to\infty$ because the factor $(\tfrac{p^\ell D}N)^{k/2}$ 
  in the error term of \eqref{Ew} will tend to $\infty$ rapidly with $k$ 
  when $\ell$ is large.

(3) It is not hard to show that $\mt(x)=
   \frac{L_p(\frac12,x,\chi)L_p(\frac12,x)}{L_p(1,\chi)}\mu_\infty(x)$, in the notation
     of \eqref{Lx}.  So the above
  result may be interpreted in a manner analogous to the remark after Theorem \ref{main}.
\end{remarks}

\begin{proof}[Proof of Corollary \ref{RRcor}]
  The limit \eqref{RRlim} holds for $\phi=X_\ell$ by \eqref{Ew}, and then
  Weierstrass approximation gives it for any continuous $\phi$.
  The rest of the proof proceeds in just the same way as that of Theorem \ref{main}.
\end{proof}

\section{Low-lying zeros II}\label{II}

Here we will use Proposition \ref{RR} to prove Theorem \ref{llz2}.  
  First we need to compute the integrals of the Chebyshev polynomials against
  the measure $\mt$ defined in Proposition \ref{RR}.

\begin{proposition} \label{Xetap}
  Let $r\ge 0$ be an integer.  Then
if $\chi(p)=-1$,
\[\int_{-\infty}^\infty X_r \,d\mt = \begin{cases} p^{-r/2}&\text{if $r$ is even},\\
0&\text{if $r$ is odd.}\end{cases}\]
If $\chi(p)=1$, then 
\[\int_{-\infty}^\infty X_r\, d\mt = (r+1)p^{-r/2}.\]
\end{proposition}

\begin{proof}
    The first assertion is well-known (\cite[p.79]{Serre}).
For the second, using \eqref{Xlid} we have
\[ \frac{p-1}{(p^{1/2}+p^{-1/2}-x)^2}=\frac{1-\frac1p}{(1-p^{-1/2}x+p^{-1})^2}
  =(1-\tfrac1p)\left[\sum_{n=0}^\infty p^{-n/2}X_n(x)\right]^2\]
\begin{equation}\label{Xeq}
=(1-\tfrac1p)\left[\sum_{j=0}^\infty X_j(x)^2p^{-j}
  +2 \sum_{m=1}^\infty\sum_{n=0}^{m-1}X_m(x)X_n(x)p^{-(m+n)/2}
\right].
\end{equation}
By the Clebsch-Gordon formula 
 (or by induction using $X_{n+1}(x)=xX_n(x)-X_{n-1}(x)$),
 we have
\[X_m(x)X_n(x)=\sum_{k=0}^n X_{m-n+2k}(x),\qquad(n\le m).\]
 So \eqref{Xeq} becomes
\begin{equation}\label{Xsum}
(1-\tfrac1p)\left[\sum_{j=0}^\infty \sum_{t=0}^j X_{2t}(x)p^{-j}
  +2 \sum_{m=1}^\infty\sum_{n=0}^{m-1}\sum_{k=0}^n X_{m-n+2k}(x)p^{-(m+n)/2}
\right].
\end{equation}
For the double sum, 
\begin{equation}\label{2sum}
\sum_{j=0}^\infty \sum_{t=0}^j X_{2t}(x)p^{-j}=\sum_{t=0}^\infty X_{2t}(x)
  \sum_{j=0}^\infty p^{-(j+t)}
=(1-\tfrac1p)^{-1}\sum_{t=0}^\infty X_{2t}(x)p^{-t}.
\end{equation}
For the triple sum, we observe that the map $(m,n,k)\mapsto (m-n+2k,m-n,m)$
defines a bijection from
\[\{(m,n,k)|\, m\ge 1,\, 0\le n\le m-1,\, 0\le k\le n\}\]
to
\[\{(u,b,m)|\, u\ge 1,\, 1\le b\le u,\, b\equiv u\mod 2,\, m\ge \tfrac{u+b}2\}
\]
with inverse $(u,b,m)\mapsto (m, m-b, \tfrac{u-b}2)$.
Therefore,
\[ \sum_{m=1}^\infty\sum_{n=0}^{m-1}\sum_{k=0}^n X_{m-n+2k}(x)p^{-(m+n)/2}
=\sum_{u=1}^\infty X_u(x)\sum_{b\equiv u\mod 2\atop{1\le b\le u}}
  \sum_{m=\frac{u+b}2}^\infty p^{-(2m-b)/2}\]
\[= \sum_{u=1}^\infty X_u(x) \sum_{b\equiv u\mod 2\atop{1\le b\le u}}
  p^{b/2}p^{-(u+b)/2} (1-\tfrac1p)^{-1}
\]
\[= (1-\tfrac1p)^{-1}\sum_{u=1}^\infty X_u(x)p^{-u/2} \sum_{b\equiv u\mod 2\atop{1\le b\le u}}
1
\]
The sum over $b$ has the value $\tfrac u2$ if $u$ is even, 
 and $\tfrac{u+1}2$ if $u$ is odd. Using this and \eqref{2sum}, \eqref{Xsum}
  becomes
\[\sum_{r\ge 0 \text{ even}} X_{r}(x)p^{-r/2} + 2\sum_{r\ge 2\text{ even}}
 \tfrac r2X_r(x)p^{-r/2}+2\sum_{r\ge 1\text{ odd}}\tfrac{r+1}2X_r(u)p^{-r/2}.
\]
In the middle sum, we can actually take $r\ge 0$ because of the $\tfrac r2$ coefficient.
This proves that
\begin{equation}\label{etap}
d\mt(x)=\sum_{r=0}^\infty(r+1)p^{-r/2}X_r(x) d\mu_\infty(x).
\end{equation}
The proposition now follows immediately using the orthonormality of the
  Chebyshev polynomials relative to $d\mu_\infty$.
\end{proof}

With this proposition in hand, we obtain the following two
  special cases of Proposition \ref{RR}.

\begin{corollary} In the notation of Proposition \ref{RR}, for any $0<\e<1$,
\begin{equation}\label{p12}
\E^w_\F(\lambda_\cdot(p)) =\begin{cases} 2p^{-1/2} 
  +O(\frac{p^{k+\frac{1}2}}{N^{k/2-\e}})
  & \text{if }\chi(p)=1\\\\
 O(\frac{p^{k+\frac{1}2}}{N^{k/2-\e}})&\text{if }\chi(p)=-1,\end{cases}
\end{equation}
    the error terms vanishing if $p<\frac ND$, and
\begin{equation}\label{p22}
\E^w_\F(\lambda_\cdot(p^2)) 
  =\begin{cases} 3p^{-1} +O(\frac{p^{2k+1}}{N^{k/2-\e}})
  & \text{if }\chi(p)=1\\\\
  p^{-1} +O(\frac{p^{2k+1}}{N^{k/2-\e}})
  & \text{if }\chi(p)=-1,\end{cases}
\end{equation}
    the error terms vanishing if $p^2<\frac ND$. 
      Implied constants depend on $k$, $D$ and $\e$.
\end{corollary}

We can now prove Theorem \ref{llz2} following the method in Section \ref{llz1}. 
Suppose $\Supp(\phi)\subset [-\alpha,\alpha]$ for some $\alpha<1$.  Then for all $N$
  sufficiently large,
\begin{equation}\label{NQ}
 Q^\alpha = N^{\alpha}k^{2\alpha} < \frac ND.
\end{equation}
  In the explicit formula, the sum \eqref{p1} involves only primes $p\le Q^\alpha$, which
    by the above means that \eqref{p12} holds with no error term.
  Therefore \eqref{p1} is equal to
\[-2\sum_{p\nmid N\atop{\chi(p)=1}}\frac{2\log p}{p\log Q}\widehat{\phi}
 (\frac{\log p}{\log Q})
=-4\sum_{p\atop{\chi(p)=1}} \frac{\log p}{p\log Q}
  \widehat{\phi}\Bigl(\frac{\log p}{\log Q}\Bigr)+
  O\Bigl(\frac{\log\log 3N}{\log Q}\Bigr).\]
Because $\chi$ is a nontrivial quadratic character, its value is 1 on exactly
  half of the primes.  By the prime number theorem for arithmetic progressions,
  the above is
\[=-\phi(0)+O\Bigl(\frac{\log\log 3N}{\log Q}\Bigr).\]

The sum \eqref{p2} involves only primes satisfying $p^2\le Q^\alpha$.  So for sufficiently
large $N$ as above, we may apply \eqref{p22} with no error term.  Substituting it into
\eqref{p2}, one obtains an expression that can be absorbed into the error term 
$O\Bigl(\frac{\log\log 3N}{\log Q}\Bigr)$, as in \eqref{Ep2}.

It now follows that if $\alpha<1$ and $N$ satisfies \eqref{NQ},  
\[\frac{\sum_{f\in \F} \D(f,\phi)w_f}
  {\sum_{f\in \F}w_f}
=\widehat{\phi}(0) -\frac12\phi(0)
+O\Bigl(\frac{\log\log 3N}{\log N}\Bigr)\]
for an implied constant depending only on $\phi$.  This proves Theorem \ref{llz2}.

We remark that if we instead fix $N$ and allow $k\to\infty$, 
   we cannot obtain the analog of Theorem \ref{llz2} by this method.
   Indeed, the contribution of the error term in \eqref{p12} to \eqref{p1} gives
  an expression involving
\begin{equation}\label{Qk}
  \sum_{p\le Q^\alpha}p^{k}\log p,
\end{equation}
which up to small powers of $k$ grows like $k^{\alpha k}$. 
  There is not enough decay in the $k$ aspect in 
  \eqref{Ew} to cancel this growth as $k\to\infty$ for any $\alpha>0$.

\section{Proof of Proposition \ref{RR}}

The papers \cite{RR} and \cite{FW} each use the relative trace formula to
  develop a formula for an average of $L$-values which in the simplest case takes
    the form
\[\sum_{f\in \F_{N,k}^{\new}} w_f \widehat{f}_p(\pi_p),\]
where: $w_f=\frac{\Lambda(\frac 12,f\times \chi)\Lambda(\frac12,f)}{\|f\|^2}$ for
  a quadratic character $\chi=\chi_{-D}$,
 $N$ is a prime not dividing $D$, $\chi(-N)=1$, $k>2$ is even,
$p\neq N$ is a prime,
  $\widehat{f}_p$ is the Satake transform
  of a compactly supported bi-$\GL_2(\Z_p)$-invariant local test function 
    $f_p: \GL_2(\Q_p)\longrightarrow \C$, and $\pi_p$ is the
  unramified local representation determined by the cusp form $f$.
  (In \cite{RR} the notation $f_p^\wedge(a_p(\varphi))$ is used,
   where $a_p(\varphi)$ corresponds with our $\lambda_f(p)$.)

   For our purpose, we need to choose the particular test function
  $f_p$ whose Satake transform is equal to the Chebyshev polynomial $X_\ell$. 
   This function is given as follows.  For $K_p=\GL_2(\Z_p)$ and $Z_p$ the center
  of $\GL_2(\Q_p)$, let 
\[M(p^\ell) = \bigcup_{i+j=\ell\atop{i\ge j\ge 0}}Z_p K_p\mat{p^i}{}{}{p^j}K_p
= \bigcup_{i+j=\ell\atop{i\ge j\ge 0}}Z_p K_p\mat{p^{i-j}}{}{}{1}K_p\]
\begin{equation}\label{Mpl}
=\bigcup_{j=0}^{\lfloor\frac\ell2\rfloor} Z_pK_p\mat{p^{\ell-2j}}{}{}1K_p.
\end{equation}
Define $f_p:\GL_2(\Q_p)\longrightarrow\C$ by
\begin{equation}\label{fp}
f_p(g)=\begin{cases} p^{-\ell/2}&\text{if }g\in M(p^\ell)\\
  0&\text{otherwise.}\end{cases}
\end{equation}

\begin{proposition}
For $f_p$ as above, and any newform $f\in S_k(N)$, let
  $\pi_p$ be the unramified principal series representation of $\GL_2(\Q_p)$
  determined by $f$.  Then
    \[\widehat{f_p}(\pi_p)=X_\ell(\lambda_f(p)).\]
\end{proposition}
\begin{proof}
  Denoting the Satake parameters of $\pi_p$ by
  $\{\alpha,{\alpha}^{-1}\}$, we have $\alpha+{\alpha}^{-1}=\lambda_f(p)$.
    By definition, $\widehat{f_p}(\pi_p)$ is the eigenvalue of the operator
  $\pi_p(f_p)$ acting on the unique $K_p$-fixed vector of $\pi_p$.
  For the moment, take $f_p$ to be the characteristic function of the set $M(p^\ell)$
  defined above. It
  is shown in \cite[Propositions 4.4-4.5]{pethil} that, in our current notation, 
  $p^{-\ell/2}\widehat{f}_p(\pi_p)=X_\ell(\lambda_f(p))$.
  Therefore, upon scaling the characteristic function by $p^{-\ell/2}$
  we get the desired result.
\end{proof}

\subsection{The theorem of Feigon and Whitehouse}\label{FWthm}

Equation \eqref{Ew2} of Proposition \ref{RR} follows immediately from the special 
  case of \cite[Theorem 6.1]{FW} given in \eqref{FW} below.
  Following \cite[\S 6.3]{FW}, we take $F=\Q$, $\Omega$ trivial,
 and $N$ prime with $N>D p^\ell$ 
    and $\chi(-N)=1$.  Then taking $f_p$ as in \eqref{fp},
\cite[Theorem 6.1]{FW} gives
  \begin{equation}\label{FW}
      \frac1{\nu(N)}\sum_{f\in \F_{N,k}^\new}w_f\, X_\ell(\lambda_f(p))
  =c_k L(1,\chi)\int_{-\infty}^\infty X_\ell\, d\mt,
  \end{equation}
  where 
  \[ c_k=\frac{2^k}{4\pi}\frac{(\frac k2-1)!^2}{(k-2)!}=
  \frac{k-1}{4\pi}2^kB(\frac k2,\frac k2)\]
  for the Beta function $B$.
 (Variants of the exact formula \eqref{FW} may also be found in \cite{MR}, \cite{FMP}, and
    \cite{ST}.)

    \begin{remarks}
    (1) We have adjusted for the fact that we have normalized the
      completed $L$-functions as in \eqref{L},
      whereas the normalization in \cite[p. 407]{FW} is twice ours.

(2) We have also adjusted for the fact that the $L$-value $L(1,\chi)$ 
  is the Dirichlet series (not completed by a Gamma factor), 
    whereas in \cite{FW} the completed $L$-value is used,
    normalized by $L_\infty(1,\chi)=(2\pi)^{-1}$ as seen in \cite[p. 407]{FW}.

(3) The lower bound for $N$ of $Dp^\ell$ comes from the definition of $\mathcal{J}(f_p)$
 found in \cite[p. 386]{FW}.  Since $p\nmid N$, we have $G(\Q_p)=\PGL_2(\Q_p)$,
   and using \eqref{Mpl} it follows that for our particular test function,
   $|\mathcal{J}(f_p)|=p^\ell$.  This matches \cite[Corollary 1]{MR}.\\
    \end{remarks}

\subsection{The theorem of Ramakrishnan and Rogawski}\label{RRthm}

As powerful as \eqref{FW} is, it is of interest in some situations to
  have a formula for the averages in which $N$ is not required to be large
   in relation to $D$ and $p^\ell$.  In this range, the error bound given in
\cite{FW} and \cite{MR} is $O(N^{-1})$ in the $N$-aspect,
  so the best bound remains that found in the original paper of Ramakrishnan and Rogawski
     who obtained \eqref{FW} up to $O(N^{-k/2+\e})$.
   By going through their calculations, we will uncover the dependence 
   of the error on both $k$ and $p$.
  The final result is given in Theorem \ref{RR2}.

With the choice of test function \eqref{fp}, the spectral side of the relative trace formula in
  \cite[Prop. 4.1]{RR} becomes
\[      \frac1{\nu(N)}\sum_{f\in \F_{N,k}^\new}w_f\, X_\ell(\lambda_f(p))
  =c_k L(1,\chi)\int_{-\infty}^\infty X_\ell\, d\mt+ I_{reg}
  \]
for $c_k$ as above, where
\[I_{reg}= \sum_{x\in\Q-\{0,1\}} I(x),\]
is the sum of the so-called regular terms, where, 
  for a certain test function $f$ whose local components will be recalled below, 
\[I(x)=\iint_{\A^*\times\A^*} f(\mat{ab}{ax}{b}1)\chi(a)^{-1} d^*a\,d^*b.\]
Here, we abuse notation and write $\chi$ for the unitary adelic Hecke character 
  determined by the Dirichlet character $\chi$ fixed earlier.
The integrals $I(x)$ are computed locally in \cite[\S2.7]{RR} and their
  sum is bounded in \S 3 of their paper.  We shall reexamine these proofs
  in order to determine the dependence on $p$ and $k$.

  The statements of \cite[Prop. 2.4abcde]{RR} each
  have errors, but this does not affect the validity of the trace formula given
  in \S5 of their paper.
  The following is a corrected version of their proposition.

\begin{proposition}\label{Ipx}  
For $x\in\Q-\{0,1\}$ and $f_v$ as in \cite{RR}, define the local integrals
\[I_v(x)=\iint_{\Q_v^*\times\Q_v^*}f_v(\mat{ab}{ax}b1)\chi_v(a)^{-1}d^*a\,d^*b.\]
  Then the following statements hold.
\begin{enumerate}
\item[(a)] Let $v=q$ be a finite prime not dividing $pND$. Then: 
\begin{itemize}
\item $I_v(x)=0$ if $v(1-x)>0$.
\item If $v(1-x)=0$ and $v(x)=0$, then $I_v(x)=1$.  
\item Generally if $v(1-x)\le 0$, then
\[|I_v(x)| \le \begin{cases}v(x)^2&\text{if }v(x)\neq 0\\1&\text{if }v(x)=0.
  \end{cases}\]
\end{itemize}
\item[(b)] Let $v=q$ be a prime dividing $D$, and write $c=v(D)\ge 1$.
  Then:
\begin{itemize}
\item $I_v(x)=0$ if $v(1-x)>c$.
\item If $v(1-x)\le c$, then 
\[|I_v(x)|\le 6q^{c/2}(2c+1 +|v(x)|)\le 6q^{c/2}(2c+1)(1+|v(x)|).\]
\end{itemize}
\item[(c)] Let $v=N$.  Then $I_v(x)=0$ unless $v(x)\ge 1$ (and hence $v(1-x)=0$).
   In this case
\[|I_N(x)|\le \nu(N)|v_N(x)|.\]
\item[(d)] Let $v=p$, and let $f_p$ be the test function defined in \eqref{fp}.
We suppose $\ell>0$ since the $\ell=0$ case is covered by (a).
    Then $I_p(x)$ vanishes unless $v(1-x)\le \ell$, in which case
\[|I_p(x)|\le 4p^{-\ell/2}\ell (\ell+1+|v(x)|)\le 4p^{-\ell/2}\ell(\ell+1)(1+|v(x)|).\]
\item[(e)] When $v=\infty$, 
\[|I_\infty(x)|\ll \frac{|1-x|^{k/2}}{|x|}\]
   for an absolute implied constant.
\end{enumerate}
\end{proposition}

\begin{proof}  
We follow the proof and notation of \cite{RR}.  
We begin with part (e), where $f_\infty(g)=d_k\ol{\sg{\pi_k(g)v,v}}$ is the matrix
  coefficient of the weight $k$ discrete series representation of $\PGL_2(\R)$ with
  lowest weight unit vector $v$ and formal degree $d_k$.
    In \cite[Prop 2.4e]{RR}, $I_\infty(x)$ is expressed in terms
  of a certain quantity $I_\infty(\e,\delta,\nu)$ which is defined as being
   independent of $x$.  This seems to be a typo; as is clear from their proof,
  $I_\infty(x)$ does depend on $x$. 
  But the proof is flawed anyhow for other reasons, so we will not try
  to correct the definition of $I_\infty(\e,\delta,\nu)$.
  For $\delta,\nu\in\{\pm 1\}$, set
\[I'_x(\delta,\nu)=\int_0^\infty\int_0^\infty\frac{a^{k/2-1}b^{k/2-1}da\,db}
  {(ax-\nu b+\delta i(ab+\nu))^k}.\]
(This is $I'_\infty(-\nu,\delta,\nu)$ in the notation of \cite{RR}.)
Following the proof in \cite{RR} (we caution that the displayed formula there
  for $f_\infty(\smat{ab}{ax}b1)$ is incorrect), we find, upon observing that $(-1)^k=1$
  since $k$ must be even, that
\begin{equation}\label{Iinf}
I_\infty(x) =\begin{cases} d_k(2i)^k(1-x)^{k/2}[I'_x(-1,1)-I'_x(1,1)]&
  \text{if }1-x>0\\
d_k(2i)^k(x-1)^{k/2}[I'_x(-1,-1)-I'_x(1,-1)]&\text{if }1-x<0.\end{cases}
\end{equation}
As shown in the proof of \cite[Lemma 7]{RR}, we have
\begin{equation}\label{Ipinf}
I'_x(\delta,\nu)={B(\tfrac{k}2,\tfrac{k}2)}(\delta i)^{k/2}J_x,
\end{equation}
where $B(x,y)$ is the Beta function, and
\[J_x=\int_0^\infty\frac{a^{k/2-1}da}{(ax+\delta\nu i)^{k/2}(a+\delta\nu i)^{k/2}}.\]
The proof in \cite{RR} now rotates the line of integration to a purely imaginary ray,
  overlooking the fact that this ray passes through poles of the 
  integrand in many cases.
  (Their proof is fixable if one assumes $x>0$, but in fact $I_\infty(x)$ need not
  vanish if $x<0$, despite the assertion to the contrary in \cite[\S3]{RR}.)
  The integral $J_x$ can presumably be computed in terms of special 
  functions even when $x<0$,
  but since ultimately this integral forms part of an error term, 
  we choose simply to bound it as follows.
  Observing that $|\tfrac a{a\pm i}|<1$ for $a>0$,
\[|J_x|\le \int_0^\infty \frac {da}{|ax\pm i|^{k/2}|a\pm i|}
\le \int_0^\infty \frac {da}{|ax\pm i|^{k/2}}
= \frac1{|x|}\int_0^\infty \frac {du}{|u\pm i|^{k/2}}\]
\[=\tfrac12 B(\tfrac12,\tfrac k4-\tfrac12)|x|^{-1}\]
by \cite[8.380.3]{GR}.
  By the above, \eqref{Iinf},
  \eqref{Ipinf}, and noting that for the standard
  measure used in \cite{RR},  $d_k=\frac{k-1}{4\pi}$ (cf. \cite[Prop. 14.4]{KL}),
  we have
\[|I_\infty(x)|\ll 2^kkB(\tfrac12,\tfrac{k-2}4)B(\tfrac k2,\tfrac k2)
   \frac{|1-x|^{k/2}}{|x|}
\]
   for an absolute implied constant.  By Stirling's formula, 
$B(\tfrac k2,\tfrac k2)\sim \frac{2\sqrt\pi}{\sqrt{\frac k2}2^k}$, and
  $B(\tfrac12,\tfrac{k-2}4)\sim \frac{2\sqrt{\pi}}{(k-2)^{1/2}}$.
 This gives $|I_\infty(x)|\ll |1-x|^{k/2}|x|^{-1}$ for an absolute
  implied constant, which proves assertion (e).

To prove (a) and (d), let 
  $q$ be a prime, fix an integer $r\ge 0$, and let $f_q$ be the 
  characteristic function of $Z_qK_q\smat{q^r}{}{}1K_q$.
Then $f_q(\smat{ab}{ax}b1)$ is nonzero if and only if there exists
  $\lambda\in\Q_q^*$ such that $\smat{\lambda ab}{\lambda ax}{\lambda b}\lambda
  \in K_q\smat{q^r}{}{}1K_q$.  By the theory
  of determinantal divisors (\cite[p. 28]{Ne}), a matrix $g\in \GL_2(\Q_q)$ 
  belongs to $K_q\mat{q^r}{}{}1K_q$ 
 if and only if each of the following holds:
\begin{itemize}
\item $\det g\in q^r\Z_q^*$
\item each entry of $g$ belongs to $\Z_q$
\item some entry of $g$ belongs to $\Z_q^*$.
\end{itemize}
(When $r=0$, the third condition is already implied by the first.)
 Therefore, $f_q(\smat{ab}{ax}b1)\neq 0$ if and only if there exists $\lambda\in\Q_q^*$
 such that:
\begin{enumerate}
\item $2v(\lambda)+v(a)+v(b)+v(1-x) =r$
\item $v(\lambda)+v(a)+v(b)\ge 0$
\item $v(\lambda)+v(a)+v(x)\ge 0$
\item $v(\lambda)+v(b)\ge 0$
\item $v(\lambda)\ge 0$
\item[(5b)] Equality occurs in at least one of (2)-(5).\\
.\hskip -1.3cm Eliminating $v(\lambda)$, we obtain the following conditions:
\item $v(a)+v(x)-v(1-x)\ge -r$ \qquad (from (2)+(3)$-$(1))
\item $v(b)-v(1-x)\ge -r$ \qquad (from (2)+(4)$-$(1))
\item $v(x)-v(1-x)\ge -r$ \qquad (from (3)+(4)$-$(1))
\item $v(1-x)\le r$ \qquad (from (1)$-$(2)$-$(5))
\item $v(a)+v(1-x)\le r$ \qquad (from (1)$-$(4)$-$(5))
\item $v(b)+v(1-x)-v(x)\le r$ \qquad (from (1)$-$(3)$-$(5))
\item $v(a)+v(b)+v(1-x)\le r$ \qquad (from (1)$-$2(5))
\item $v(b)\ge v(a)+v(1-x)-r$ \qquad (from 2(4)$-$(1)).
\end{enumerate}
This leads to the following condensed set of conditions, the last of which 
  is from (5b) and was overlooked in the proof of \cite[Prop. 2.4]{RR}:
\begin{enumerate}
\item[(i)] $v(1-x)\le r$
\item[(ii)] $v(x)\ge v(1-x)-r$
\item[(iii)] $v(1-x)-v(x)-r\le v(a)\le\min\{r-v(1-x),r-v(1-x)-v(b)\}$
\item[(iv)] $\max\{v(1-x)-r,v(a)+v(1-x)-r\}\le v(b)\le v(x)+r-v(1-x)$
\item[(v)] At least one of the following holds:
\begin{enumerate}
\item[(va)] $v(a)+v(b)+v(1-x)=r$ \qquad(if $v(\lambda)=0$, using (1))
\item[(vb)] $v(a)+v(b)-v(1-x)=-r$ \qquad(if (2)=0, using 2(2)$-$(1))
\item[(vc)] $v(a)-v(b)+2v(x)-v(1-x)=-r$ \qquad(if (3)=0, using 2(3)$-$(1))
\item[(vd)] $v(b)-v(a)-v(1-x)=-r$ \qquad(if (4)=0, using 2(4)$-$(1)).
\end{enumerate}
\end{enumerate}

We may now prove part (a).  Suppose $q\nmid pND$.  
  Then $f_q$ is the characteristic function of $K_q$ and
  we can take $r=0$ in the above discussion. The first part of (a) follows from (i).
If $r=v(x)=v(1-x)=0$, we see from (iii) and (iv) that
  $v(a)=v(b)=0$, and since $\chi_q$ is unramified and $\meas(\Z_q^*)=1$,
   it follows that $I_v(x)=1$.  Now suppose $v(1-x)<0$.  Then
  $v(x)=v(1-x)$, and (iii) and (iv) become
\[0\le v(a)\le -v(x),\qquad v(x)\le v(b)\le 0.\]
Using the fact that $\chi_q$ is unramified and $\meas(\Z_q^*)=1$, we find
\[|I_v(x)|\le \sum_{m=0}^{-v(x)}\sum_{n=v(x)}^0 1,\]
 and the last
  assertion of (a) follows in this case.  Likewise, if $v(1-x)=0$,
  then $v(x)\ge 0$, and (iii) and (iv) become
\[-v(x)\le v(a)\le 0,\qquad 0\le v(b)\le v(x),\]
and the assertion holds in this case as well.  This proves (a).

Before proving (d), we make some observations about the above
  conditions for general $r\ge 0$.
  If $v(1-x)\le r$,
we see from (v) that once $v(a)$ is fixed,
  there are at most {\em four} possibilities for $v(b)$.  
Setting $m=v(a)$ and $n=v(b)$, we immediately see that
\[|I_v(x)|\le \sum_{m=v(1-x)-r-v(x)}^{r-v(1-x)}\,\,
   \sum_{n \in \{4\text{ values}\}}1 =4\Bigl(2r-2v(1-x)+v(x)+1\Bigr).\]
Observing that if $v(1-x)>0$ (resp. $v(1-x)=0$, resp. $v(1-x)<0$) then
  $v(x)=0$ (resp. $v(x)\ge 0$, resp. $v(x)=v(1-x)$), it follows easily that
in all cases,
\begin{equation}\label{Iv}
|I_v(x)| \le 4\Bigl(2r+1+|v(x)|\Bigr)\le 8(r+1+|v(x)|).
\end{equation}

Now suppose $q=p$ and $f_p$ is the test function defined in \eqref{fp}.
Then by the above, $I_p(x)$ vanishes if $v(1-x)>\ell$.  When $v(1-x)\le \ell$,
  by \eqref{Mpl}, \eqref{fp} and \eqref{Iv}, we have
\[|I_p(x)|\le p^{-\ell/2}\sum_{j=0}^{\lfloor \frac{\ell}2\rfloor}8(\ell-2j+1+|v(x)|)
\le p^{-\ell/2}\tfrac{8\ell}2(\ell+1+|v(x)|).\]
This proves (d).

Next, consider $v=N$.  Then for 
\[K_0(N)_N=\{\smat abcd\in\GL_2(\Z_N)|\, c\in N\Z_N\},\]
  $f_N$ is the characteristic function of $Z_NK_0(N)_N$,
  scaled by $\nu(N)$. So $f_N(\smat{ab}{ax}b1)\neq 0$ if and only if there 
  exists $\lambda\in\Q_N^*$ such that 
$\smat{\lambda ab}{\lambda ax}{\lambda b}{\lambda}\in K_0(N)_N$.
  The lower right entry must be a unit, which means that in fact we may take $\lambda=1$.
  Therefore 
\[\smat{ab}{ax}b1\in K_0(N)_N,\]
 which means:
\begin{enumerate}
\item[(1$''$)] $v(a)+v(b)+v(1-x)=0$
\item[(2$''$)] $v(a)+v(b)=0$
\item[(3$''$)] $v(a)+v(x)\ge 0$
\item[(4$''$)] $v(b)\ge 1.$
\end{enumerate}
As a result, the integrand vanishes unless:
\begin{itemize}
\item $v(1-x)=0$
\item $v(a)=-v(b)\le -1$
\item $v(x)\ge 1$.
\end{itemize}
It follows that $I_N(x)=0$ unless $v(x)\ge 1$, in which case
\[\ds|I_N(x)|\le \sum_{m=-v(x)}^{-1}\nu(N),\]
which proves (c).

Lastly, take $v=q$ to be a prime divisor of $D$, and set $c=v(D)\ge 1$.
There are some oversights in the definition of the local test function $f_q$
  at such a place in \cite[p. 706]{RR}: the notation $\chi_{1,v}$ is not defined, 
  $\chi_v$ does not define a character of the additive group $X$, and
  it is asserted that the integral $g(\chi_v)$ defined there, which 
  clearly has absolute value $\le 1$, coincides with the classical Gauss 
  sum which has absolute value $q^{c/2}$.
  A detailed treatment of the local test function with the desired spectral
  properties (and giving the same main term on the geometric side in \cite{RR})
  is given in \cite[(3.11)-(3.12)]{twists}. 
  For our purpose, it is enough to know that 
\[\Supp(f_q)=\bigcup_{m\mod D\Z_q\atop{q\nmid m}}\smat1{-m/D}01Z_qK_q,\]
and $f_q=\sum_m f_{m,q}$, where $f_{m,q}$ is supported on the coset
  indexed by $m$ and has absolute value $q^{-c/2}$ there.

To match the notation in \cite{RR}, let $z=m/D$ (so $v(z)=-c$)
   and write $f_{z,v}$ for $f_{m,q}$.
  Then $f_{z,v}(\smat{ab}{ax}b1)\neq 0$ if and only if there exists $\lambda\in\Q_q^*$
  such that
\[\mat{\lambda}{}{}\lambda\mat1z01\mat{ab}{ax}b1
  =\mat{\lambda b(a+z)}{\lambda(ax+z)}{\lambda b}\lambda\in K_q.\]
Thus,
\begin{enumerate}
\item[($1'$)] $2v(\lambda)+v(a)+v(b)+v(1-x)=0$
\item[($2'$)] $v(\lambda)+v(a+z)+v(b)\ge 0$
\item[($3'$)] $v(\lambda)+v(ax+z)\ge 0$
\item[($4'$)] $v(\lambda)+v(b)\ge 0$
\item[($5'$)] $v(\lambda)\ge 0$
\item[($5'$b)] Equality holds in at least one of ($2'$)-($5'$).\\
.\hskip -1.25cm As before, we eliminate $v(\lambda)$ to get the following:
\item[($6'$)] $v(a+z)+v(ax+z)-v(a)-v(1-x)\ge 0\qquad$ (from ($2'$)+($3'$)$-$($1'$))
\item[($7'$)] $v(b)+v(a+z)-v(a)\ge v(1-x)\qquad$ (from ($2'$)+($4'$)$-$($1'$))
\item[($8'$)] $v(ax+z)-v(a)-v(1-x)\ge 0\qquad$ (from ($3'$)+($4'$)$-$($1'$))
\item[($9'$)] $v(1-x)\le v(a+z)-v(a)\qquad$ (from ($1'$)$-$($2'$)$-$($5'$))
\item[($10'$)] $v(a)+v(1-x)\le 0\qquad$ (from ($1'$)$-$($4'$)$-$($5'$))
\item[($11'$)] $v(a)+v(b)+v(1-x)-v(ax+z)\le 0\qquad$ (from ($1'$)$-$($3'$)$-$($5'$))
\item[($12'$)] $v(a)+v(b)+v(1-x)\le 0\qquad$ (from ($1'$)$-$2($5'$))
\item[($13'$)] $v(b)\ge v(a)+v(1-x)\qquad$ (from 2($4'$)$-$($1'$)).
\end{enumerate}
(Only (11$'$) differs from the list in \cite{RR}, whose (11$'$) seems to be
  an unmodified paste from (11).)
We claim that the above implies the following set of conditions:
\begin{enumerate}
\item[(x)] $v(1-x)\le c$
\item[(y)] $v(1-x)-c\le v(b)\le v(x)+c-v(1-x)$
\item[(z)] At least one of the following holds:
\begin{enumerate}
\item[(zi)] $v(a)=-v(1-x)-v(b)\qquad$ (if $v(\lambda)=0$, using ($1'$))
\item[(zii)] $v(a)+v(1-x)-2v(a+z)-v(b)=0\qquad$ (if ($2'$)=0, using ($1'$)$-$2($2'$))
\item[(ziii)] $v(a)+v(b)+v(1-x)-2v(ax+z)=0\qquad$ (if ($3'$)=0, using ($1'$)$-$2(3$'$))
\item[(ziv)] $v(a)-v(b)+v(1-x)=0\qquad$ (if ($4'$)=0, using ($1'$)$-$2(4$'$)).
\end{enumerate}
\end{enumerate}
It suffices to prove (x) and (y) since (z) follows from ($5'b$).
To prove (x), if $v(a)\neq v(z)$, then $v(a+z)=\min\{v(a),v(z)\}$, so $v(a+z)-v(a)\le 0$,
  which, by ($9'$), gives $v(1-x)\le 0<c$.  On the other hand, if $v(a)=v(z)=-c$, then by 
  ($10'$), $v(1-x)\le c$, as needed.

For (y), note that if $v(a)=v(z)=-c$, then ($13'$) gives $v(1-x)-c\le v(b)$ in that
  case.  If $v(a)\neq v(z)$, then as before $v(a+z)-v(a)\le 0$, and ($7'$) then gives
  $v(1-x)-c<v(1-x)\le v(b)$.  This proves the lower bound in (y).
  For the upper bound, suppose first that $v(ax)\neq v(z)$.
  Then $v(ax+z)=\min\{v(ax),v(z)\}$, so $v(ax+z)\le v(a)+v(x)$.
  ($11'$) then gives $v(b)\le v(x)-v(1-x)$, which is stronger than the desired 
  upper bound.  If $v(ax)=v(z)$, then ($11'$) is not helpful because 
  $v(ax+z)=\infty$ is possible.  However, in this case $v(ax)=v(a)+v(x)=-c$, so 
  ($12'$) gives $v(b)\le v(x)+c-v(1-x)$, as needed.
  
Finally, we claim that once $v(b)$ is fixed, there are at most {\em six}
  possible values of $v(a)$ for which (z) is satisfied.
  It suffices to show that there are at most two possibilities for $v(a)$
  if (zii) (resp. (ziii)) is satisfied.  Suppose $a$ and $\tilde{a}$ have
  different valuations and each satisfy (zii). We claim that 
  $v(\tilde{a})=2v(z)-v(a)$.
   Write $\tilde{a}=q^t ua$ for $u\in\Z_q^*$ and some integer $t\neq 0$.  Then
\[v(\tilde{a})-2v(\tilde{a}+z)=v(a)-2v(a+z),\]
which gives
\[v(q^tua+z)=v(a+z)+\tfrac t2.\]
By Lemma \ref{alem} below, we get $t=2v(z)-2v(a)$, as claimed.
  For (ziii), by the same argument we get
\[v(q^tuax+z)=v(ax+z)+\tfrac t2,\]
so $t$ is again determined by Lemma \ref{alem}: $t=2v(z)-2v(ax)$.

By the above discussion, summing over $z$ (i.e. over $m\in (\Z_q/D\Z_q)^*$), and
  using $|f_{z,v}(g)|=q^{-c/2}$ if nonzero, when $v(1-x)\le c$ we have
\[|I_q(x)| \le q^{-c/2}\varphi(q^c)\sum_{n=v(1-x)-c}^{v(x)+c-v(1-x)}
  \sum_{\{6\text{ values}\}}1\le q^{c/2}6\bigl(v(x)+2c-2v(1-x)+1\bigr)\]
\[\le 6q^{c/2}(2c+1+|v(x)|),\]
where the latter inequality is obtained by considering the cases 
  $v(1-x)$ being greater than, equal to, or less than $0$.
  This proves part (b) of the proposition.
\end{proof}

\begin{lemma}\label{alem}
   Let $a,z\in\Q_q^*$ with $a+z\neq 0$, and
suppose there exist $u\in\Z_q^*$ and $t$ a nonzero integer such that
\begin{equation}\label{k}
v(q^tua+z)=v(a+z)+\tfrac t2
\end{equation}
    where $v=v_q$.  Then $t=2v(z)-2v(a)$.
\end{lemma}

\begin{proof}
    First suppose $v(z)=0$.  We need to show that $t=-2v(a)$.
    If $v(a)=0$ too,
  then $v(a+z)\ge 0$, and \eqref{k} leads to a contradiction if either
  $t>0$ or $t<0$.
Suppose $v(a)>0$, so that $v(a+z)=0$.
    If $t>-v(a)$,  then \eqref{k} becomes
  \[0=v(q^tua+z)=\tfrac t2,\]
a contradiction. If $t=-v(a)$, then \eqref{k} becomes
\[0\le v(q^tua+z)=\tfrac t2\]
which is also a contradiction.  If $t<-v(a)$, then \eqref{k} becomes
\[t+v(a)=v(q^tua+z)= \tfrac{t}2,\]
which gives $t=-2v(a)$.  A similar analysis gives the same conclusion if $v(a)<0$.

 In the general case, write $z=q^cw$ for $w\in\Z_q^*$.  Factoring out $q^c$,
 \eqref{k} becomes
    \[  v(q^tu\tfrac a{q^c}+w)=v(\tfrac a{q^c}+w)+\tfrac t2.\]
    The special case discussed above then gives $t=-2v(a/q^c)=2v(z)-2v(a)$,
    as needed.
\end{proof}

\begin{proposition}
With local components $f_v$ as in Proposition \ref{Ipx}, 
the sum of the regular terms is
\[I_{reg}\ll \frac{\nu(N)D^{k}}
{N^{k/2-\e}} {p^{\ell(k+\frac{1}2)}},\]
for any $0<\e<1$, where the implied constant depends only on $\ell$, $D$, and $\e$.
\end{proposition}

\begin{proof}
We closely follow \cite[\S3]{RR}.  Let $M=Dp^\ell$.
Suppose $I(x)\neq 0$.  Then by Proposition \ref{Ipx}, 
  $v_q(1-x)\le v_q(M)$ for all primes $q$.
   This means that $n:=\frac{M}{1-x}\in \Z$.
The map $x\mapsto \frac1{1-x}$ is a bijection from $\Q-\{0,1\}$ to itself.  
  Therefore $n$ is not equal to $0$ or $M$. 
  Since $N\nmid M$ and $v_N(1-x)=0$ by Proposition \ref{Ipx}c,
   we have
\[v_N(n-M)=v_N(M(\tfrac1{1-x}-1))=v_N(\tfrac{x}{1-x})=v_N(x)\ge 1,\]
where the latter inequality is again from Proposition \ref{Ipx}c.
Thus $N|(n-M)$. Note that $x=\frac{n-M}n$.
So
\begin{equation}\label{Ireg}
I_{reg}=\sum_{n\in M+N\Z,\atop{n\neq 0,M}} I\Bigl(\frac{n-M}n\Bigr).
\end{equation}
Since $N\nmid M$, the condition $n\neq 0$ is superfluous.
  As mentioned earlier, the assertion in \cite[\S3]{RR} that 
  $I_\infty(x)=0$ if $x<0$ is incorrect.
Now by Proposition \ref{Ipx},
\begin{equation}\label{InM}
I(\tfrac{n-M}n)\ll p^{-\ell/2}\sqrt{D}\nu(N)|I_\infty(\tfrac{n-M}n)|
  \prod_{q|n(n-M)}B_q(\tfrac{n-M}n),
\end{equation}
where
\[B_q(\tfrac{n-M}n)=\begin{cases} 1 &\text{if }v_q(\tfrac{n-M}n)=0\text{ and }q\nmid pDN\\
  v_q(\tfrac{n-M}n)^2&\text{if }v_q(\tfrac{n-M}{n})\neq 0\text{ and }q\nmid pDN\\
|v_q(\tfrac{n-M}n)|&\text{if }q=N\\
  (1+|v_q(\tfrac{n-M}n)|)&\text{if } q|M(=p^\ell D),\end{cases}\]
  and the implied constant in \eqref{InM} depends only on $\ell$ and $D$.

For the archimedean part, by Proposition \ref{Ipx}e we have
\[|I_\infty(\tfrac{n-M}n)|\ll |1-\tfrac{n-M}n|^{k/2}\cdot|\tfrac{n-M}n|^{-1} 
  =\frac{M^{k/2}}{|n|^{k/2}|1-\tfrac{M}n|}.\]
Observe that for fixed $M$, $|1-\tfrac Mn|$ is as small as possible when $n=M+1$
  since $n\neq M$.  Hence $|1-\tfrac Mn| \ge \tfrac1{M+1}$.  So
for an absolute implied constant,
\begin{equation}\label{Iinfb}
|I_\infty(\tfrac{n-M}n)|\ll \frac{M^{k/2+1}}{|n|^{k/2}}.
\end{equation}

To treat the product in \eqref{InM}, as shown in the proof of
\cite[Lemma 8]{RR}, for any $\e>0$, there exists a constant 
$C$ depending only on $\e$ such that
\[\prod_{q|n(n-M)}|v_q(\tfrac{n-M}n)|\le C|n|^\e|n-M|^\e\]
for all $n\neq M$.  This is in turn 
\[\ll |n|^\e|nM|^\e\ll |n|^\e M^\e,\]
where not all epsilons are the same but each may be made
arbitrarily small.
It follows similarly that
\begin{equation}\label{Bb}
\prod_{q|n(n-M)}B_q(\tfrac{n-M}n)\ll |n|^\e M^\e
\end{equation}
for any $\e>0$.

  Using \eqref{Iinfb} and \eqref{Bb} and recalling that $M=p^\ell D$, 
  \eqref{InM} gives
\[|I(\tfrac{n-M}n)|\ll_{D,\ell,\e} 
  p^{-\frac \ell2}p^{\ell(\frac k2+1+\e)}D^{k/2}\nu(N)\frac1{|n|^{k/2-\e}}.
\]
So
\[I_{reg}\ll p^{\ell(\frac{k+1}2+\e)}D^{k/2}\nu(N)\sum_{\text{nonzero } m\in\Z}
  \frac1{|M+Nm|^{k/2-\e}}\]
\begin{equation}\label{Isum}
=\frac{p^{\ell(\frac{k+1}2+\e)}D^{k/2}\nu(N)}{N^{k/2-\e}}
\sum_{m\neq 0}\frac1{|m+\frac{M}N|^{k/2-\e}}.
\end{equation}
Noting that $\tfrac MN\notin\Z$ and $k\ge 4$, the sum is convergent when $\e<1$.
We will show that this sum is $O(M^{k/2-\e})$.

Generally, for $a>1$ and a noninteger $u>0$ with $u=\lfloor u\rfloor+\{u\}$,
\[\sum_{m\in\Z}\frac1{|m+u|^a}= \{u\}^{-a}+(1-\{u\})^{-a} + 
  \sum_{m\ge 1}\frac1{|m+\{u\}|^a}+\sum_{m\le -2} \frac1{|m+\{u\}|^a}\]
\begin{equation}\label{msum}
\le \{u\}^{-a}+(1-\{u\})^{-a}+2\sum_{m\ge 1}\frac1{m^a}.
\end{equation}
We will apply this with $u=\tfrac{M}N$. 
If $N<M$, then writing $M=qN+r$, we see that $\{\tfrac MN\}=\tfrac rN\ge \frac1{M-1}$.
  Likewise $1-\{\tfrac MN\}\ge \tfrac1{M-1}$, so
  $\{\tfrac{M}N\}^{-a}+(1-\{\tfrac MN\})^{-a}\le 2(M-1)^a$.
If $M<N$, then $\{\tfrac MN\}=\tfrac MN$, and the first term in \eqref{msum} comes 
  from $m=0$, which is excluded in \eqref{Isum}.  For the second term,
  $(1-\{\tfrac MN\})^{-a}\le (1-\frac{M}{M+1})^{-a}=(M+1)^a$.
Taking $a=\tfrac k2-\e$, the third term in \eqref{msum} is $2\zeta(\tfrac k2-\e)\le 2\zeta(2-\e)$.
It follows that for any prime $N\nmid M$,
\[\sum_{m\neq 0}\frac1{|m+\frac{M}N|^{k/2-\e}}\ll M^{k/2-\e},\]
as claimed.
With $M=p^\ell D$, \eqref{Isum} now yields
\[I_{reg}\ll_{D,\ell,\e} p^{\ell(k+\frac12)} D^{k}\frac{\nu(N)}{N^{k/2 -\e}}.\qedhere\]
\end{proof}

By what we have shown, along with the
  computation of the main term and measure in \cite{RR}, upon dividing through 
  by $\nu(N)$ we obtain the following.

\begin{theorem} \label{RR2} Let $k>2$ be an even integer, $\chi=\chi_{-D}$ be 
  as in Theorem \ref{llz2},
  $N$ a prime not dividing $D$ with $\chi(-N)=-1$, and $p$ a prime not
  dividing $ND$.  Then for all $\ell\ge 0$ and $0<\e<1$,
\begin{align*}
\frac1{\nu(N)}\sum_{f\in\F_{N,k}^{\new}}&\frac{\Lambda(\frac12,f)
  \Lambda(\frac12,f\times\chi)}{\|f\|^2} X_\ell(\lambda_f(p))\\
 &= 
  c_kL(1,\chi)\int_{-\infty}^\infty X_\ell\, d\mt 
  + O\Bigl(\frac{p^{\ell(k+\frac{1}2)}D^{k}}{N^{k/2-\e}}\Bigr)
\end{align*}
where $c_k= \frac{k-1}{4\pi}2^kB(\tfrac k2,\tfrac k2)$, and
the implied constant depends only on $\ell$, $D$, and $\e$.
\end{theorem}

\begin{remarks}
(1) In \cite{RR}, the formula for the formal degree of the weight $k$
  discrete series of $\PGL_2(\R)$ is given as $d_k=\frac{k-1}2$.
  This should be corrected to $d_k=\frac{k-1}{4\pi}$, which
  corresponds to the Haar measure on $\SL_2(\R)=\mathbf{H}\times \SO(2)$
   determined by the measure $\frac{dx \,dy}{y^2}$ on $\mathbf{H}$ 
   and the measure on $\SO(2)$ of total length $1$.

(2) 
For $p$ fixed as above, if $\int X_\ell \,d\mt\neq 0$, 
we see from \eqref{Ew2} that the sum on the left-hand side in Theorem \ref{RR2}
  is nonzero when $N>p^\ell D$. This is stronger than what can be deduced from the above
  using \eqref{ck} below.
\end{remarks}

We may now prove \eqref{Ew}, and so complete the proof of Proposition \ref{RR}.
   By Theorem \ref{RR2}, we have
\[\frac1{\nu(N)}\sum_{f\in\F}w_fX_\ell(\lambda_f(p))=F_\ell +E_\ell,\]
where $F_\ell$ is the main term and 
  $E_\ell \ll p^{\ell(k+\frac{1}2)}C_0$, where $C_0=\frac{D^{k}}{N^{k/2-\e}}$.
By the proof of Proposition \ref{Xl},
\[\frac{\sum_{f\in\F}w_fX_\ell(\lambda_f(p))}{\sum_{f\in\F}w_f}
=\int_{-\infty}^\infty X_\ell\,d\mt+O\Bigl(p^{\ell(k+\frac{1}2)}
  \frac{\frac{C_0}{F_0}}{1+\frac{E_0}{F_0}}\Bigr).\]
(Note that $F_0\neq 0$ since $\int X_0 \,d\mt =1$ as shown in 
  Proposition \ref{Xetap}.)
As noted earlier,
$2^k B(\tfrac k2,\tfrac k2)\sim \frac{2\sqrt{2\pi}}{\sqrt{k}}$, so that
\begin{equation}\label{ck}
c_k=\frac{k-1}{4\pi}2^k B(\tfrac k2,\tfrac k2)\sim \sqrt{\frac{k}{2\pi}}.
\end{equation}
Now
\[\frac{E_0}{F_0}\ll \frac{C_0}{F_0} = \frac{D^{k}}{c_kL(1,\chi) N^{k/2-\e}}
  \ll \frac{D^{k}}{k^{1/2} N^{k/2-\e}},\]
    and \eqref{Ew} follows. \\

\noindent{\bf Acknowledgements:} We would like to thank the referee for carefully
  reading the manuscript and offering many suggestions that improved the quality of the
  results and exposition. We also thank 
  Charles Li, Steven J. Miller, Nigel Pitt, Chip Snyder, and Fan Zhou
  for numerous helpful conversations.

\end{document}